\documentclass[a4paper,12pt]{article}
\usepackage[T1]{fontenc}
\usepackage{iwona,amsmath,amsthm,amsfonts,amssymb,microtype,thmtools,underscore,mathtools,anyfontsize,float} 
\usepackage[usenames,dvipsnames,svgnames,table]{xcolor}
\usepackage[unicode=true]{hyperref}
\hypersetup{%
colorlinks,
breaklinks=true,
linkcolor={black},
citecolor={black},
urlcolor={blue!60!black},
pdftitle={Tree-Partitions and Small-Spread Tree-Decompositions},
pdfauthor={Distel--Kaul--Kaul--Wood}
}
\usepackage[noabbrev,capitalise]{cleveref}
\crefname{lem}{Lemma}{Lemmas}
\crefname{thm}{Theorem}{Theorems}
\crefname{cor}{Corollary}{Corollaries}
\crefname{prop}{Proposition}{Propositions}
\crefname{conj}{Conjecture}{Conjectures}
\crefname{ques}{Question}{Questions}
\crefname{rmk}{Remark}{Remarks}
\crefname{openproblem}{Open Problem}{Open Problems}
\crefformat{equation}{(#2#1#3)}
\Crefformat{equation}{Equation #2(#1)#3}

\usepackage[shortlabels]{enumitem}
\setlist[itemize]{topsep=0ex,itemsep=0ex,parsep=0.4ex}
\setlist[enumerate]{topsep=0ex,itemsep=0ex,parsep=0.4ex}
\usepackage[longnamesfirst,numbers,sort&compress]{natbib}
\makeatletter
\def\NAT@spacechar{~}
\makeatother
\usepackage[margin=30mm]{geometry}
\renewcommand{\baselinestretch}{1.065}
\allowdisplaybreaks
\usepackage{tikz}
\usetikzlibrary{positioning}
\usetikzlibrary{patterns}
\newcommand{\defn}[1]{\textcolor{Maroon}{\emph{#1}}}
\newcommand{\mathdefn}[1]{\textcolor{Maroon}{#1}}

\DeclarePairedDelimiter{\ceil}{\lceil}{\rceil}

\renewcommand{\geq}{\geqslant}
\renewcommand{\leq}{\leqslant}
\renewcommand{\emptyset}{\varnothing}
\newcommand{\TT}{\mathcal{T}}
\DeclareMathOperator{\dist}{dist}

\DeclareMathOperator{\tw}{tw}
\DeclareMathOperator{\dtw}{dtw}
\DeclareMathOperator{\tpw}{tpw}
\DeclareMathOperator{\pw}{pw}

\renewcommand{\thefootnote}{\fnsymbol{footnote}}
\allowdisplaybreaks
\theoremstyle{plain}
\newtheorem{thm}{Theorem}
\newtheorem{lem}[thm]{Lemma}
\newtheorem{cor}[thm]{Corollary}
\newtheorem{prop}[thm]{Proposition}
\newtheorem{ques}[thm]{Question}

\newtheorem{obs}[thm]{Observation}
\theoremstyle{definition}

\begin{document}
\title{\bf\Large Tree-Partitions and Small-Spread Tree-Decompositions\footnote{This work was initiated at the \href{https://www.matrix-inst.org.au/events/structural-graph-theory-downunder-ll/}{Structural Graph Theory Downunder II} workshop at the Mathematical Research Institute MATRIX (March 2022). Preliminary versions of this paper, only with the results for tree-partitions (and with weaker constants than in \cref{ImprovedTreePartition}), were published as 
``Tree-Partitions with Bounded Degree Trees'' in the 
\href{https://www.matrix-inst.org.au/2021-matrix-annals/}{2021-22 MATRIX Annals}, and as ``Tree-Partitions with Small Bounded Degree Trees'' on the arXiv~\citep{DW22a}.}}
\author{Marc Distel\,\footnotemark[3] 
\qquad Neel Kaul\,\footnotemark[3]  
\qquad Raj Kaul\,\footnotemark[3]  
\qquad David~R.~Wood\,\footnotemark[3]}
\footnotetext[3]{School of Mathematics, Monash University, Melbourne, Australia (\texttt{\{marc.distel,neel.kaul,raj.kaul,david.wood\}@monash.edu}). Research of D.W.\ supported by the Australian Research Council and by NSERC. Research of M.D., N.K.\ and R.K.\ is supported by an Australian Government Research Training Program Scholarship.}

\maketitle
\vspace*{-2ex}


\begin{abstract}
Tree-decompositions and treewidth are of fundamental importance in structural and algorithmic graph theory. The \defn{spread} of a tree-decomposition is the minimum integer $s$ such that every vertex lies in at most $s$ bags. A tree-decomposition is \defn{domino} if it has spread 2, which is the smallest interesting value of spread. So that spread 1 becomes interesting, one can relax the definition of tree-decomposition to \defn{tree-partition}, which allows the endpoints of each edge to be in the same bag or adjacent bags, while demanding that each vertex appears in exactly one bag. Ding and Oporowski [1995] showed that every graph $G$ with treewidth $k$ and maximum degree $\Delta$ has a tree-partition with width $O(k\Delta)$. We prove the same result with an improved constant, and with the extra property that the underlying tree has maximum degree $O(\Delta)$ and $O(|V(G)|/k\Delta)$ vertices. This result implies (with an improved constant) the best known upper bound on the domino treewidth of $O(k\Delta^2)$, due to Bodlaender [1999]. Moreover, solving an open problem of Bodlaender, we show this upper bound is best possible, by exhibiting graphs with domino treewidth $\Omega(k\Delta^2)$ for $k\geq 2$. On the other hand, allowing the spread to be a function of $k$, we show that width $O(k\Delta)$ can be achieved. This result exploits a connection to chordal completions, which we show is best possible, a result of independent interest.
\end{abstract}

\renewcommand{\thefootnote}{\arabic{footnote}}

\newpage

\section{Introduction}
\label{Intro}

\subsection{Tree-Decompositions} 

A \defn{tree-decomposition} of a graph $G$ is a collection $(B_x\subseteq V(G):x\in V(T))$ of subsets of $V(G)$ (called \defn{bags}) indexed by the nodes of a tree $T$, such that: (a) for every edge $uv\in E(G)$, some bag $B_x$ contains both $u$ and $v$; and (b) for every vertex $v\in V(G)$, the set $\{x\in V(T):v\in B_x\}$ induces a non-empty subtree of $T$. Properties (a) and (b) are respectively called the `edge-property' and `vertex-property' of tree-decompositions. The \defn{width} of a tree-decomposition is the maximum size of a bag,
minus $1$. The \defn{treewidth} of a graph $G$, denoted by \defn{$\tw(G)$}, is the minimum width of a tree-decomposition of $G$. Treewidth is the standard measure of how similar a graph is to a tree. Indeed, a connected graph has treewidth 1 if and only if it is a tree. Treewidth is of fundamental importance in structural and algorithmic graph theory; see \citep{Reed03,HW17,Bodlaender98} for surveys.

While width is the most important property of tree-decompositions, several other properties are also studied, such as
chromatic number of the bags~\citep{Seymour16,HK17,BFMMSTT19,HRWY21}, independence number of the bags~\citep{Yolov18,DMS21}, diameter of the bags \citep{CDN16,Lokshtanov10,BS24,DG07},
or treewidth of the bags~\citep{LNW,HW25c}.

We start by considering the following property. The \defn{spread} of a vertex $v$ in a tree-decomposition is the number of bags that contain $v$. Several papers have studied the spread of vertices in tree-decompositions~\citep{Wood25,DO95,BG26,BodEng-JAlg97,Norin26} and in path-decompositions~\citep{DM04a,DM04b,DM05,DM07} (where it is called \defn{persistence}). 

Note that in a tree-decomposition of width $k$, if a vertex $v$ has spread $s$ then $\deg(v)\leq ks$. Thus, it is necessary that the width or the spread depends on (the maximum) degree. This paper explores trade-offs between the spread and the width in tree-decompositions of graphs with given treewidth and given maximum degree. 

In a tree-decomposition of a connected graph, if every vertex has spread 1, then the edge-property implies that there is only one bag. So spread 2 is the smallest interesting value. We return to tree-decompositions with spread 2 in \cref{IntroSmallSpread}. The next definition relaxes the edge-property of tree-decompositions by allowing adjacent vertices to be in adjacent bags, so that `spread 1' becomes interesting.

\subsection{Tree-Partitions}

For a graph~$G$ and a tree~$T$, a \defn{$T$-partition} of~$G$ is a partition~${(V_x \colon x\in V(T))}$ of~${V(G)}$ indexed by the nodes of~$T$, such that for every edge~${vw}$ of~$G$, if~${v \in V_x}$ and~${w \in V_y}$, then~${x = y}$ or~${xy \in E(T)}$. The \defn{width} of a $T$-partition is~${\max\{ |{V_x}| \colon x \in V(T)\}}$. The \defn{tree-partition-width}\footnote{Tree-partition-width has also been called \defn{strong treewidth} \citep{BodEng-JAlg97,Seese85}.} of a graph $G$, denoted by \defn{$\tpw(G)$}, is the minimum width of a tree-partition of $G$. 

Tree-partitions were independently introduced by \citet{Seese85} and \citet{Halin91}, and have since been widely investigated \citep{Bodlaender-DMTCS99,BodEng-JAlg97,DO95,DO96,Edenbrandt86, Wood06,Wood09,BGJ22}. Applications of tree-partitions include graph drawing~\citep{CDMW08,GLM05,DMW05,DSW07,WT07}, 
graphs of linear growth~\citep{CDGHHHMW23}, 
nonrepetitive graph colouring~\citep{BW08}, 
clustered graph colouring~\citep{ADOV03,LO18}, 
monadic second-order logic~\citep{KuskeLohrey05}, 
network emulations~\citep{Bodlaender-IPL88, Bodlaender-IC90, BvL-IC86, FF82}, 
size Ramsey numbers~\citep{DKCPS,KLWY21}, 
and the edge-{E}rd{\H{o}}s-{P}{\'o}sa property~\citep {RT17,GKRT16,CRST18}. 
The key property in all these applications is that each vertex appears only once in the tree-partition (unlike in tree-decompositions). Tree-partitions are also related to graph product structure theory since a graph $G$ has a $T$-partition of width at most $k$ if and only if $G$ is isomorphic to a subgraph of $T\boxtimes K_k$ for some tree $T$; see~\citep{UTW,DJMMW24,DEMWW22} for example. 

Bounded tree-partition-width implies bounded treewidth, as noted by \citet{Seese85}. In particular, for every graph~$G$,
\[\tw(G) \leq 2\tpw(G)-1.\]
Of course, ${\tw(T) = \tpw(T) = 1}$ for every tree~$T$. But in general, $\tpw(G)$ can be much larger than~$\tw(G)$. For example, fan graphs on~$n$ vertices have treewidth~2 and tree-partition-width~$\Omega(\sqrt{n})$. 
On the other hand, the referee of \citep{DO95} showed that if the maximum degree and treewidth are both bounded, then so is the tree-partition-width, which is one of the most useful results about tree-partitions\footnote{\cref{TreeProduct} is stated in \cite{DO95} with ``$\tw(G)$'' instead of ``$\tw(G)+1$'', but a close inspection of the proof shows that ``$\tw(G)+1$'' is needed.}. A graph $G$ is \defn{non-trivial} if $E(G)\not=\emptyset$. Let \defn{$\Delta(G)$} be the maximum degree of a vertex of $G$. 
 
\begin{thm}[\cite{DO95}]
\label{TreeProduct}
For any non-trivial graph $G$,
$$\tpw(G) \leq 24 (\tw(G)+1)\Delta(G).$$
\end{thm}

\citet{Wood09} showed that \cref{TreeProduct} is best possible up to the multiplicative constant, and also improved the constant $24$ to $9+6\sqrt{2}\approx 17.48$. 

Our first result improves the constant in \cref{TreeProduct} to 8, where the tree $T$ indexing the tree-partition has two extra properties, namely that $\Delta(T)$ and $|V(T)|$ are small. 

\begin{restatable}{thm}{ImprovedTreePartition}
\label{ImprovedTreePartition}
For any integers $k,d\geq 1$, every graph with treewidth at most $k-1$ and  maximum degree at most $d$ has a $T$-partition of width at most $k(8d-3)$, for some tree $T$ with $\Delta(T)\leq 4d-1$ and $|V(T)|\leq \ceil{\frac{|V(G)|}{kd}}$. In particular, for any non-trivial graph $G$,
\[ \tpw(G) \leq (\tw(G)+1)(8\Delta(G)-3).\]
\end{restatable}

We now elaborate on the two extra properties of tree-partitions in \cref{ImprovedTreePartition}. 

First consider the maximum degree of $T$: For any tree-partition $(B_x:x\in V(T))$ of a graph $G$ with width $k$, for each node $x\in V(T)$, there are at most $\sum_{v\in B_x}\deg(v)$ edges between $B_x$ and $G-B_x$. Thus we may assume that $\deg_T(x)\leq |B_x|\Delta(G)\leq k\Delta(G)$, otherwise delete an `unused' edge of $T$ and add an edge to $T$ between leaf vertices of the resulting component subtrees. It follows that if $\tpw(G)\leq k$, then $G$ has a $T$-partition of width at most $k$ for some tree $T$ with maximum degree at most $\max\{k\Delta(G),2\}$. By \cref{TreeProduct}, every graph $G$ has a $T$-partition of width at most $24 (\tw(G)+1)\Delta(G)$ for some tree $T$ with maximum degree at most $24 (\tw(G)+1)\Delta(G)^2$. This fact has been used in several applications of \cref{TreeProduct} (see \citep{CDMW08,DSW07} for example). \cref{ImprovedTreePartition}, which is proved in \cref{TreePartitionProof},  enables a $\tw(G)\Delta(G)^2$ term to be replaced by a $\Delta(G)$ term in the aforementioned applications. Also note that the linear upper bound on $\Delta(T)$ in \cref{ImprovedTreePartition} is best possible even for trees (see \cref{LowerBound} in \cref{TPW-LowerBound}).

Now consider the second property, the number of vertices in $T$. This property is motivated by results about size-Ramsey numbers by \citet{DKCPS}, where tree-partitions of $n$-vertex graphs with $\tw(G)\in O(\sqrt{n})$ play a critical role, and it is essential that $|V(T)| \ll |V(G)|$ and $\Delta(T)$ is independent of $\tw(G)$. 
In \cref{ImprovedTreePartition}, 
$|V(T)|\ll |V(G)|$ when $\tw(G)$ is unbounded, and $\Delta(T)$ is independent of $\tw(G)$. Thus \cref{ImprovedTreePartition} is suitable for such applications to size-Ramsey numbers, in particular, \cref{ImprovedTreePartition} improves upon a similar result of \citet[Lemma~2.1]{DKCPS} which has an extra $O(\log n)$ factor in the width. 
Finally, note that the bound, $|V(T)| \leq O( \frac{|V(G)|}{kd})$, in \cref{ImprovedTreePartition} is best possible if the width is $O(kd)$, since $|V(G)|\leq w|V(T)|$ in any $T$-partition of width~$w$.


Here we give an example of \cref{ImprovedTreePartition}. \citet{AST90} proved that every $K_t$-minor-free $n$-vertex graph $G$ has treewidth at most  $t^{3/2}n^{1/2}-1$. \cref{ImprovedTreePartition} thus implies:

\begin{cor}
\label{MinorClosedClass}
Every $K_t$-minor-free $n$-vertex graph with maximum degree $\Delta$ has a $T$-partition of width at most $8t^{3/2}\Delta n^{1/2}$, for some tree $T$ with $\Delta(T)\leq 4\Delta-1$ and $|V(T)|\leq n^{1/2}/\Delta t^{3/2}$.
\end{cor}


\subsection{Tree-Decompositions with Small Spread}
\label{IntroSmallSpread}

A tree-decomposition is \defn{domino} if every vertex has spread at most 2. The \defn{domino treewidth} of a graph $G$, denoted \defn{$\dtw(G)$}, is the minimum width of a domino tree-decomposition of $G$. See \citep{DO95,Bodlaender-DMTCS99,BodEng-JAlg97,Wood09} for work on domino tree-decompositions. \citet{DO95} and \citet{BodEng-JAlg97} both proved that graphs of bounded treewidth $k$ and bounded maximum degree $\Delta$ have bounded domino treewidth. The previously best known bound is $\dtw(G)\leq (9k+7)\Delta(\Delta+1)-1$ by \citet{Bodlaender-DMTCS99}.

\citet{BodEng-JAlg97} observed that domino tree-decompositions can be used to construct tree-partitions; in particular, $\tpw(G)\leq \dtw(G) +1$. Conversely, we now show that tree-partitions can be used to construct  domino tree-decompositions. Let $(B_x:x\in V(T))$ be a tree-partition of $G$. Root $T$ at an arbitrary node. For each $x\in V(T)$ let $B'_x$ be the set of vertices $w\in V(G)$ such that either $w\in B_x$, or $w$ has a neighbour in $B_x$ and $w$ is in $B_y$ for some child $y$ of $x$ in $T$. Note that $|B'_x|\leq |B_x| (\Delta(G)+1)$, and $(B'_x:x\in V(T))$ is  a domino tree-decomposition of $G$. Thus
\begin{align}
\label{NewMakeDomino}
\dtw(G) \leq \tpw(G) ( \Delta(G)+1) -1.
\end{align}

%




\cref{ImprovedTreePartition} and \eqref{NewMakeDomino} imply that every graph with treewidth $k$ and maximum degree $\Delta$ satisfies 
\begin{equation}
    \label{ImprovedDTW}
    \dtw(G) \leq (k+1)(\Delta+1)(8\Delta-3)-1,
\end{equation}
which is a small constant factor improvement over the above-mentioned bound due to \citet{Bodlaender-DMTCS99}. 


Since every tree indexes a tree-partition of itself with width 1, \cref{NewMakeDomino} implies that $\dtw(T)\leq\Delta(G)$ for every tree $T$. In contrast, for treewidth at least 2, we show that quadratic dependence on $\Delta$ is essential in any upper bound on the domino treewidth.

\begin{thm}
\label{DTWLowerBound}
For any integers $k\geq 2$ and $\Delta\geq k+7$ there exists a graph $G$ with treewidth at most $k$, maximum degree at most $\Delta$, and domino treewidth $\Omega(k\Delta^2)$.
\end{thm}

\cref{DTWLowerBound}, which is proved in \cref{DominoTreewidthLowerBound}, solves an open problem of  \citet{Bodlaender-DMTCS99}, who asked whether the above-mentioned $O(k\Delta^2)$ upper bound is best possible. \cref{DTWLowerBound} says the answer is `yes'. The best previous lower bound was $\Omega(k\Delta)$ \citep{Bodlaender-DMTCS99,Wood09}.

The graph in the $k=2$ case of \cref{DTWLowerBound} is outerplanar (and the proof shows that this case is the heart of the matter). This says that quadratic dependence on $\Delta$ is essential in any upper bound on the domino treewidth, even for outerplanar graphs, which are one of the simplest classes beyond trees. On the other hand, for outerplanar graphs we show that the quadratic dependence on $\Delta$ in \cref{DTWLowerBound} is peculiar to spread 2:





\begin{thm}
\label{OuterplanarUpperBound}
Every outerplanar graph $G$ with maximum degree $\Delta\geq 1$ has a tree-decomposition with width at most $6\Delta-5$ such that each vertex has spread at most 3.
\end{thm}

We generalise the $O(\Delta)$ width bound in \cref{OuterplanarUpperBound} for arbitrary graphs of bounded treewidth as follows.

\begin{restatable}{thm}{LinearWidthBoundedSpread}
\label{LinearWidthBoundedSpread}
Every graph $G$ with treewidth $k$ and maximum degree $\Delta$ has a tree-decomposition with width $O(k\Delta)$, such that each vertex in $G$ has spread $2^{O(k)}$.
\end{restatable}

The proof of \cref{LinearWidthBoundedSpread} depends on a result about chordal completions with small maximum degree, due to \citet{Wood25}. We show that this degree bound is best possible, which is of independent interest. This material is presented in \cref{ChordalCompletions}. The proofs of \cref{OuterplanarUpperBound,LinearWidthBoundedSpread} are completed in \cref{BoundedSpread}.

\cref{LinearWidthBoundedSpread} sits between domino treewidth (spread 2) and the results where spread is allowed to depend on degree. \citet{DO95} first proved that every graph with bounded treewidth has a tree-decomposition of bounded width, where each vertex $v$ has spread $O(\deg(v))$. In this result the hidden dependence on treewidth is exponential (or greater). Much better bounds were obtained by \citet{Wood25}:

\begin{thm}[\protect\citep{Wood25}]
\label{SmallSpread}
Every graph $G$ with treewidth $k$ has a tree-decomposition with width at most $14k+13$ such that each vertex $v$ has spread at most $\deg_G(v)+1$. 
\end{thm}

The method of \citet{Wood25} was pushed further by \citet{BG26}, improving the width bound in \cref{SmallSpread} from  $14k+13$ to $(3+\epsilon)k$ (for any fixed $\epsilon>0$), at the expense of a constant factor increase in the spread bound. 

The following table summarises the known upper bounds on width and spread in a tree-decomposition of a graph with treewidth $k$ and maximum degree $\Delta$. Results in grey are superseded by subsequent improved bounds.

\setlength{\tabcolsep}{1ex}
\begin{center}
\begin{tabular}{ccl}
\hline
&\\[-2.5ex]
$\text{width}\leq$ & $\text{spread}(v)\leq$ & Reference\\[0.5ex]
\hline
&\\[-2ex]
\textcolor{lightgray}{$2^{k+1}(k+1)-1$} & \textcolor{lightgray}{$1+3^{2^k}\,\deg(v)$} & \textcolor{lightgray}{\hypersetup{citecolor=lightgray}\citet{DO95}}\\[0.5ex]
$14k+13$ & $\deg(v)+1$ & \citet{Wood25}\\[0.5ex]
$(3+\epsilon)k$ & $O_\epsilon(\deg(v))$ & \citet{BG26}\\[0.5ex]
$O(k\Delta)$ & $2^{O(k)}$ & \cref{LinearWidthBoundedSpread}\\[0.5ex]
\textcolor{lightgray}{$f(k,\Delta)$} & \textcolor{lightgray}{$2$} & \textcolor{lightgray}{\hypersetup{citecolor=lightgray}\citet{BodEng-JAlg97}}\\[0.5ex]
\textcolor{lightgray}{$O(k^2\Delta^3)$} & \textcolor{lightgray}{$2$} & \textcolor{lightgray}{\hypersetup{citecolor=lightgray}\citet{DO95}}\\[0.5ex]
\textcolor{lightgray}{$(9+o(1))k\Delta^2$} & \textcolor{lightgray}{$2$} & 
\textcolor{lightgray}{\hypersetup{citecolor=lightgray}\citet{Bodlaender-DMTCS99} }\\[0.5ex]
$(8+o(1))k\Delta^2$ & $2$ & Equation~\cref{ImprovedDTW}\\[0.5ex]
\hline
\end{tabular}
\end{center}

\section{Tree-Partition Proofs}
\label{TreePartitionProof}

This section proves \cref{ImprovedTreePartition}, which constructs tree-partitions of small width, where the underlying tree has small degree and order. The proof relies on the following two lemmas. (When combined the lemmas are similar to Corollary~2.3 of \citet{Bodlaender-DMTCS99}.)\ 

\begin{lem}[{\protect\citep[Corollary~11]{Wood25}}]
\label{TreewidthSetSep}
For any tree-decomposition $\TT$ of a graph $G$, for any set $S\subseteq V(G)$, for any integer $b\geq 1$, there is a set $X\subseteq V(G)$, equal to the union of $b$ bags in $\TT$, such that each component of $G-X$ has at most $\frac{|S|}{b+1}$ vertices in $S$. 
\end{lem}

For a graph $G$ and set $S\subseteq V(G)$, let $\mathdefn{N_G(S)}:=(\bigcup_{v\in S}N_G(v))\setminus S$. 

\begin{lem}[{\protect\citep[Lemma~8]{DHHJLMMRW25}}]
\label{TwoBags}
For any graph $G$ and tree-decomposition $\TT$ of $G$, for any integer $b\geq1$, if $X$ is the union of $b$ bags of $\TT$, then there is a set $Y\subseteq V(G)$ that is the union of at most $2b-1$ bags of $\TT$ such that $X \subseteq Y$ and for every component $C$ of $G-Y$, $N(V(C)) \cap Y$ is a subset of the union of at most two bags of $\TT$. 
\end{lem}

The next lemma is the heart of the proof of \cref{ImprovedTreePartition}. A tree-partition $(B_x:x\in V(G))$ is \defn{rooted} at $r\in V(T)$ if $T$ is rooted at $r$. 

\begin{lem}\label{lem: NewTreePartitionWidthBound}
For any integers $k,d\geq1$, for any graph $G$ with treewidth at most $k-1$ and maximum degree at most $d$, for any $S\subseteq V(G)$ with $ |S| \leq 4kd$, 
there exists a tree-partition $(B_x:x\in V(T))$ of $G$ rooted at $r\in V(T)$ such that:
\begin{enumerate}[(a)] 
    \item\label{lem: NewTreePartitionWidthBound (a)} $|B_x|\leq k(8d-3)$ for each $x\in V(T)$, 
    \item\label{lem: NewTreePartitionWidthBound (b)} $S\subseteq B_r$, 
    \item\label{lem: NewTreePartitionWidthBound (c)} each node of $T$ has at most $4d-2$ children, 
    \item\label{lem: NewTreePartitionWidthBound (d)} for each $x\in V(T)$ if $|B_x|\leq 2kd$ then $x$ is a leaf of $T$ (meaning $x$ has no children), and
    \item\label{lem: NewTreePartitionWidthBound (e)} for each $x\in V(T)$, there is at most one child $y$ of $x$ in $T$ with $|B_y|\leq 2kd$. 
    \end{enumerate}
\end{lem}

\begin{proof}
We proceed by induction on $|V(G)|$. 
In the base case, if $|V(G)|\leq 4kd$ then the tree-partition with one bag $B_r:=V(G)$ satisfies the claim (in particular, satisfying \labelcref{lem: NewTreePartitionWidthBound (d)}). 
Now assume that $|V(G)|\geq 4kd$.
If $|S|<4kd$ then add vertices to $S$ so that $|S|=4kd$. 

Let $\TT$ be a tree-decomposition of $G$ with width at most $k-1$.
By \cref{TreewidthSetSep} with $b=2d-1$ and \cref{TwoBags}, 
there is a set $Y\subseteq V(G)$, consisting of the union of $4d-3$ bags in $\TT$, such that for each component $C$ of $G-Y$, we have 
$|S\cap V(C)| \leq \frac{|S|}{b+1} = 2k$, and there is a set $Y_C$ contained in the union of two bags of $\TT$, such that
$N(V(C)) \cap Y = Y_C$.  
So $|Y_C|\leq 2k$. 
Let $S_C$ be the set of vertices in $C-S$ that are adjacent to at least one vertex in $S\cup Y$.
So $S_C$ is contained in the neighbourhood of $(S\cap V(C)) \cup Y_C$, implying $|S_C|\leq (2k+2k)d=4kd$. 

For any subgraph $Q$ of $G-(S\cup Y)$, let $\gamma(Q):= |V(Q)\cap N_G(S\cup Y)|$. 
Each component $Q$ of $G-(S\cup Y)$ is a subgraph of some component $C$ of $G-Y$, so $\gamma(Q)\leq|S_{C}|\leq4kd$,  and $\gamma(G-(S\cup Y))\leq d|S\cup Y|$. Define a \defn{pseudo-component} to be any union of components of $G-(S\cup Y)$. Let $G_1,\dots,G_m$ be a partition of $G-(S\cup Y)$ into pseudo-components, such that $\gamma(G_i)\leq 4kd$ for each $i\in\{1,\dots,m\}$, and $m$ is minimum. This is well-defined since the components of $G-(S\cup Y)$ are candidates. 

For each $i\in\{1,\dots,m\}$, let $S_i:= V(G_i)\cap N_G(S\cup Y)$, which has size $\gamma(G_i)\leq 4kd$. 
By induction, $G_i$ has a tree-partition $(B_x:x\in V(T_i))$ rooted at $r_i\in V(T_i)$ with properties \labelcref{lem: NewTreePartitionWidthBound (a)}--\labelcref{lem: NewTreePartitionWidthBound (e)} above. 
Let $T$ be the tree obtained from the disjoint union $T_1\cup\dots\cup T_m$ by adding a new root node $r$ adjacent to $r_1,\dots,r_m$. 
Define $B_r:= S\cup Y$. Since $N_G(S\cup Y)\subseteq B_{r_1}\cup\dots\cup B_{r_m}$, we have that $(B_x:x\in V(T))$ is a tree-partition of $G$. 

Property \labelcref{lem: NewTreePartitionWidthBound (a)} holds, since it holds for each $T_i$ and 
$|B_r|\leq |S|+|Y| \leq 4kd + (4d-3)k = k(8d -3)$. Property \labelcref{lem: NewTreePartitionWidthBound (b)} holds since $S\subseteq B_r$ by construction. 

We now prove  \labelcref{lem: NewTreePartitionWidthBound (c)}  for $T$. Since \labelcref{lem: NewTreePartitionWidthBound (c)} holds for each $T_i$, and $r$ has $m$ children, it suffices to show that $m\leq 4d-2$. 
By the minimality of $m$, we have $\gamma(G_i)+\gamma(G_j)>4kd$ for distinct $i,j\in\{1,\dots,m\}$. Thus
\[ (m-1)\, \gamma(G-(S\cup Y)) = (m-1)\sum_{i\in[m]}\gamma(G_i) = \sum_{1\leq i<j\leq m} (\gamma(G_i)+\gamma(G_j)) > \binom{m}{2} 4kd.\]
Hence 
\[ m <\frac{\gamma(G-(S\cup Y))}{2kd} \leq \frac{d|S\cup Y|}{2kd}  \leq \frac{dk(8d-3)}{2kd} = 4d - \frac32.\]
Therefore $m\leq 4d-2$, as desired. 

Now we prove \labelcref{lem: NewTreePartitionWidthBound (d)} for $T$. Suppose that $|B_x|\leq 2kd$ for some $x\in V(T)$. 
Since \labelcref{lem: NewTreePartitionWidthBound (d)} holds for $T_i$, if $x\in V(T_i)$ then $x$ is a leaf of $T_i$, implying that $x$ is a leaf of $T$.
Since $|B_r|\geq |S|>2kd$, property \labelcref{lem: NewTreePartitionWidthBound (d)} holds for $T$. 

Now we prove \labelcref{lem: NewTreePartitionWidthBound (e)} for $T$. Consider $x\in V(T)$. If $x\in V(T_i)$ then there is at most one child $y$ of $x$ in $T_i$ with $|B_y|\leq 2kd$, since \labelcref{lem: NewTreePartitionWidthBound (e)} holds in $T_i$. Now consider $x=r$. Suppose for the sake of contradiction that there are distinct children $y$ and $z$ of $r$ in $T$ with $|B_y|\leq 2kd$ and $|B_z|\leq 2kd$. So $y$ is the root of $T_i$, and $z$ is the root of $T_j$, for some distinct $i,j\in\{1,\dots,m\}$. By property \labelcref{lem: NewTreePartitionWidthBound (d)} for $T_i$, $y$ is a leaf of $T_i$. Thus $\gamma(G_i)=|S_i|\leq |V(G_i)|=|B_y|\leq 2kd$. Similarly, 
$\gamma(G_j)\leq 2kd$. Hence $\gamma(G_i)+\gamma(G_j)\leq 4kd$, contradicting that $\gamma(G_i)+\gamma(G_j)>4kd$. This proves property \labelcref{lem: NewTreePartitionWidthBound (e)} for $T$.
\end{proof}

%
%

\ImprovedTreePartition*

\begin{proof}
Let $(B_x:x\in V(T))$ be the tree-partition from \cref{lem: NewTreePartitionWidthBound}. By \labelcref{lem: NewTreePartitionWidthBound (a)},  the width is at most $k(8d-3)$. 
Since each node of $T$ has at most $4d-2$ children, $\Delta(T)\leq 4d-1$.
It remains to show that $|V(T)|\leq \ceil{\frac{|V(G)|}{kd}}$. 
Say $x\in V(T)$ is \defn{small} if $|B_x|\leq 2kd$. 
If the root node $r$ is small, then $r$ has no children by \labelcref{lem: NewTreePartitionWidthBound (d)}, implying $|V(T)|=1$, as desired. 
Now assume that $r$ is not small. 
Let $s$ be the number of small nodes in $T$. 
By property \labelcref{lem: NewTreePartitionWidthBound (d)}, each small node $x$ is a leaf of $T$; let $x'$ be the parent of $x$ in $T$, which is not a leaf and thus not small. 
By property \labelcref{lem: NewTreePartitionWidthBound (e)}, each node of $T$ has at most one small child. So $\{xx':x\text{ is small}\}$ is a matching in $T$, 
implying $s\leq |V(T)|/2$ and at least $|V(T)|/2$ nodes of $T$ are not small. 
Hence $|V(G)| \geq 2kd\,|V(T)|/2=kd|V(T)|$ and $|V(T)|\leq \frac{|V(G)|}{kd}$. 
\end{proof}

\subsection{Lower Bound}
\label{TPW-LowerBound}

We now prove a lower bound on the degree of the underlying tree in a tree-partition of a tree. 
For integers $\Delta\geq2$ and $d\geq1$, let $X_{\Delta,d}$ be the tree rooted at a vertex $r$ such that every leaf is at distance $d$ from $r$ and every non-leaf vertex has degree $\Delta$. Observe that $X_{\Delta,d}$ has the maximum number of vertices in a tree with maximum degree $\Delta$ and radius $d$, where  
\[ |V(X_{\Delta,d})| 
= 1+\Delta\sum_{i=0}^{d-1}(\Delta-1)^i.\]
Note that $|V(X_{2,d})|= 2d+1$, and
 if $\Delta\geq 3$ then 
\[ (\Delta-1)^d \leq |V(X_{\Delta,d})| 
=1+\tfrac{\Delta}{\Delta-2}\big((\Delta-1)^d-1\big)
\leq 3(\Delta-1)^d.\]

\begin{restatable}{prop}{LowerBound}
\label{LowerBound}
For any integer $\Delta\geq 3$, there exists $\alpha>0$ such that there are infinitely many trees $X$ with maximum degree $\Delta$ such that for every tree $T$ with maximum degree less than $\Delta$, every $T$-partition of $X$ has width at least $|V(X)|^\alpha$. Moreover, if $\Delta=3$ then $\alpha$ can be taken to be arbitrarily close to $1$. 
\end{restatable}

\begin{proof}
First suppose that $\Delta\geq 4$. 
Let $d_0\geq1$ be a sufficiently large integer so that $(\frac{\Delta-1}{\Delta-2})^{d_0}>3$. 
Let $\alpha:= {1- \log_{\Delta-1}(3^{1/d_0}(\Delta-2))}$, which is positive by the choice of $d_0$. 
Let $d$ be an integer with $d\geq d_{0}$. 
It follows that 
$(\Delta-1)^{(1-\alpha)d} \geq 3(\Delta-2)^d $. 
Consider any tree-partition $(B_u:u\in V(T))$ of $X_{\Delta,d}$, where $T$ is any tree with maximum degree at most $\Delta-1$. Let $z$ be the vertex of $T$ such that the root $r\in B_z$. Since adjacent vertices in $X_{\Delta,d}$ belong to adjacent parts or the same part in $T$, every vertex in $T$ is at distance at most $d$ from $z$. Thus $T$ has radius at most $d$, and \[|V(T)|\leq|V(X_{\Delta-1,d})| 
  \leq 3(\Delta-2)^d 
  \leq (\Delta-1)^{(1-\alpha)d} \leq
|V(X_{\Delta,d})|^{1-\alpha}  .\]
By the pigeon-hole principle, there is a vertex $u\in V(T)$ such that 
$|B_u|\geq \frac{|V(X_{\Delta,d})|}{|V(T)|}\geq |V(X_{\Delta,d})|^\alpha$.

Now assume that $\Delta=3$. Let $\alpha\in(0,1)$, let $d_0\geq1$ be a sufficiently large integer so that $2d_0+1 \leq 2^{(1-\alpha)d_0}$, and let $d$ be an integer with $d\geq d_0$.
So $2d+1 \leq 2^{(1-\alpha)d}$. Consider any tree-partition $(B_u:u\in V(T))$ of $X_{3,d}$, where $T$ is any tree with maximum degree at most $2$. By the argument above, $T$ has radius at most $d$, implying \[|V(T)|\leq|V(X_{2,d})|=2d+1 \leq 2^{(1-\alpha)d} \leq |V(X_{3,d})|^{1-\alpha}.\]
Again, there is a vertex $u\in V(T)$ such that 
$|B_u|\geq \frac{|V(X_{3,d})|}{|V(T)|}\geq |V(X_{3,d})|^{\alpha}$.
\end{proof}

\cref{LowerBound} shows that the $\Delta(T)\in O(\Delta(G))$ bound in \cref{ImprovedTreePartition} is best possible, even when $G$ is a tree. 

\section{Domino Treewidth Lower Bound}
\label{DominoTreewidthLowerBound}

This section proves \cref{DTWLowerBound}, which says that for any integers $k\geq 2$ and $\Delta\geq k+7$ there exists a graph $G$ with treewidth at most $k$, maximum degree at most $\Delta$, and domino treewidth $\Omega(k\Delta^2)$.

For integers $d\geq 2$ and $n\geq 0$, as illustrated in \cref{OuterplanarConstruction}, recursively define the following graph \defn{$G_{d,n}$} with $V(G_{d,n}) = V_0\cup \cdots\cup V_n$.  Let $G_{d,0}:=K_2$ and $V_0:=V(G_{d,0})$. Now suppose that $G_{d,i}$ and $V_0, \dots, V_i$ have been constructed. To construct $G_{d,i+1}$ from $G_{d,i}$, for every edge $uv$ of $G_{d,i}[V_i]$, add a $(2d-1)$-vertex path $P$ and make the first $d$ vertices of $P$ adjacent to $u$, and make the last $d$ vertices of $P$ adjacent to $v$. Let $V_{i+1}:=V(G_{d,i+1})\setminus V(G_{d,i})$.

\input{OuterplanarConstruction}

Suppose $V_0=\{u,v\}$. For $n\geq 1$, let $w$ be the unique common neighbour of $u$ and $v$ in $G_{d,n}$, and let $C$ be the component of $G_{d,n}-\{v,w\}$ that contains $u$. Let $\text{\defn{$H_{d,n}$}}:=G_{d,n}-V(C)$.

For integers $d\geq 2$, $n\geq 1$ and $i\in \{0, \dots, n\}$, define
\[\text{\defn{$L_i(G_{d,n})$}}:=V_i \text{ and } \text{\defn{$L_i(H_{d,n})$}}:=V(H_{d,n})\cap V_i.\]

\begin{lem}\label{lem : H}
Let $d\geq 2$ and $n\geq 2$ be integers, and let $H:=H_{d,n}$. If $(B_x : x\in V(T))$ is a domino tree-decomposition of $H$ and there exists a node $b\in V(T)$ such that
\[L_0(H) \cup L_1(H) \subseteq B_b,\]
then at least one of the following hold:
\begin{enumerate}[label=(\roman*)]
    \item\label{lem : H (i)} $L_1(H)$ lies in the intersection of two bags,
    \item\label{lem : H (ii)} there exists $p\in L_1(H)$ such that $\{p\} \cup (N_H(p) \cap L_2(H)) \subseteq B_b$, or
    \item\label{lem : H (iii)} there exists an edge $u'v'\in E(H[L_1(H)])$ and distinct nodes $a',c'\in V(T)\setminus \{b\}$ such that $u'\in B_{a'}$, $\{u',v'\}\subseteq B_b$, and $v'\in B_{c'}$.
\end{enumerate}
\end{lem}

\begin{proof}
By the definition of $H$, $L_1(H)$ induces a path $(x_1, \dots, x_d)$.

If $x_1$ lies in the bag at exactly one node, then $x_1\in B_b$ and the edge-property of tree-decompositions implies that $\{x_1\}\cup N_H(x_1) \subseteq B_b$, so \labelcref{lem : H (ii)} holds. Hence it may be assumed that there exists $a'\in V(T)\setminus \{b\}$ such that $x_1\in B_{a'}$. Let $j\in [d]$ be maximum such that $\{x_1, \dots, x_j\}\subseteq B_{a'}$. For each $i\in [j]$ we have $x_i\in B_{a'}\cap B_b$. Therefore if $j=d$, then $L_1(H)\subseteq B_{a'}\cap B_b$ and \labelcref{lem : H (i)} holds.

Now suppose $j<d$. If there exists $c'\in V(T)\setminus \{a',b\}$ such that $x_{j+1}\in B_{c'}$, then \labelcref{lem : H (iii)} holds with $u':=x_j$ and $v':=x_{j+1}$. On the other hand, we have that $x_{j+1}\not\in B_t$ for all $t\in V(T)\setminus \{a',b\}$. Moreover, $x_{j+1}\not\in B_{a'}$ by maximality of $j$. Hence $x_{j+1}$ lies in the bag at exactly one node, namely $b$. Then $\{x_{j+1}\}\cup N_H(x_{j+1})\subseteq B_b$ and \labelcref{lem : H (ii)} holds. This concludes the proof of the lemma.
\end{proof}

\begin{lem}\label{lem : G}
Let $d\geq 2$ and $n\geq 2$ be integers, and let $G:=G_{d,n}$ with $L_0(G)=\{u,v\}$. If $(B_x : x\in V(T))$ is a domino tree-decomposition of $G$ and there exists distinct nodes $a,b,c\in V(T)$ such that
\[u\in B_a,\quad \{u,v\}\subseteq B_b,\quad  v\in B_c,\]
then at least one of the following hold:
\begin{enumerate}[label=(\roman*)]
    \item\label{lem : G (i)} $\{u\}\cup N_G(u) \subseteq B_b$ or $\{v\}\cup N_G(v) \subseteq B_b$,
    \item\label{lem : G (ii)} there exists $p\in L_1(G)$ such that $\{p\} \cup (N_G(p) \cap L_2(G)) \subseteq B_b$, or
    \item\label{lem : G (iii)} there exists an edge $u'v'\in E(G[L_1(G)])$ and distinct nodes $a',c'\in V(T)\setminus \{b\}$ such that $u'\in B_{a'}$, $\{u',v'\}\subseteq B_b$, and $v'\in B_{c'}$.
\end{enumerate}
\end{lem}

\begin{proof}
By the definition of $G$, $L_1(G)$ induces a path $(x_1, \dots, x_{2d-1})$ such that $N_G(u) = \{v, x_1, \dots, x_d\}$ and $N_G(v)=\{u, x_d, \dots, x_{2d-1}\}$.

It is well-known that in every tree-decomposition, each clique is in some bag (see \citep[Corollary~12.3.5]{Diestel5}). Then since $u$ only appears in the bags $B_a$ and $B_b$, for all $i\in [d-1]$, the triangle $\{u, x_i, x_{i+1}\}$ lies in $B_a$ or in $B_b$. Similarly, for all $j\in \{d, \dots, 2d-2\}$, the triangle $\{v,x_j, x_{j+1}\}$ lies in $B_b$ or in $B_c$. If $\{u,x_i, x_{i+1}\}\subseteq B_b$ for all $i\in [d-1]$, then \labelcref{lem : G (i)} holds. Similarly, if $\{v,x_j, x_{j+1}\}\subseteq B_b$ for all $j\in \{d, \dots, 2d-2\}$, then \labelcref{lem : G (i)} holds. Hence it may be assumed that there exists $i\in [d-1]$ and $j\in \{d, \dots, 2d-2\}$ such that $\{u,x_i, x_{i+1}\}\subseteq B_a$ and $\{v,x_j, x_{j+1}\}\subseteq B_c$. Consider the maximum such $i$ and minimum such $j$. If $i+1=d=j$, then $x_d\in B_a\cap B_c$. However, since $a,b,c$ are distinct, $u\in B_a\cap B_b$ and $v\in B_b\cap B_c$ together imply that the triangle $\{u,v,x_d\}$ lies in $B_b$, thus $x_d\in B_a\cap B_b\cap B_c$, a contradiction. Consequently, $i<d-1$ or $j>d$. The following only analyses the former case as the latter case is symmetric: By choice of $i$, $\{u,x_i, x_{i+1}\}\subseteq B_a$, $\{u, x_{i+1}, x_{i+2}\}\subseteq B_b$, and $x_{i+2}\not\in B_a$. If there exists $c'\in V(T)\setminus \{a,b\}$ such that $x_{i+2}\in B_{c'}$, then \labelcref{lem : G (iii)} holds with $u':=x_{i+1}$, $v':=x_{i+2}$, and $a':=a$. On the other hand, we have that $x_{i+2}\not\in B_t$ for all $t\in V(T)\setminus \{b\}$, so $x_{i+2}$ lies in the bag at exactly one node, namely $b$. Then by the edge-property of tree-decompositions, $\{x_{i+2}\}\cup N_G(x_{i+2}) \subseteq B_b$ and \labelcref{lem : G (ii)} holds. This concludes the proof of the lemma.
\end{proof}

The recursive definition of $G_{d,n}$ yields the following observations about $G_{d,n}$ and $H_{d,n}$ that are key to the proof of the next lemma.

\begin{obs}\label{obs : G' hanging off}
Let $d\geq 2$ and $n\geq 2$ be integers, let $D \in \{G_{d,n}, H_{d,n}\}$, and let $u'v'\in E(D[L_1(D)])$. Let $C$ be the component of $D-\{u',v'\}$ that contains $L_0(D)$. Then, as illustrated in \cref{fig : G' hanging off}, $G':=D-V(C)$ is a copy of $G_{d,n-1}$ such that $L_0(G') = \{u',v'\}$ and $V(G') \cap L_0(D) = \emptyset$. We say that $G'$ is the copy of $G_{d,n-1}$ in $D$ hanging off $u'v'$.
\end{obs}

\begin{figure}[H]
\centering

\newcommand{\blobcolour}{LightGray} 
\newcommand{\hatchingcolour}{red}

\begin{tikzpicture}[vertex/.style={circle,draw,fill=black,inner sep=1.5pt},/tikz/xscale=1.6, /tikz/yscale=1.6]
%
%

\foreach \i in {0,2,4,6}{
    \draw[densely dashed, fill=\blobcolour, opacity=.5] (\i,1) to [out=90,in=180] (\i+0.5, 2.25) to [out=0,in=90] (\i+1,1.5) to (\i,1);
    
    \draw[densely dashed, fill=\blobcolour, opacity=.5] (\i+2,1) to [out=90,in=0] (\i+1.5, 2.25) to [out=180,in=90] (\i+1,1.5) to (\i+2,1);
        
    \draw (\i,1) to (\i+2,1);
    \draw (\i,1) to (\i+1,1.5);
    \draw (\i+2,1) to (\i+1,1.5);
    }

%
%

\node[vertex] (u) at (2.75,0.2) {};
\node[vertex] (v) at (5.25,0.2) {};

\draw (u) to (v);

\draw (u) to (0,1);
\draw (u) to (2,1);
\draw (u) to (4,1);

\draw (v) to (4,1);
\draw (v) to (6,1);
\draw (v) to (8,1);

%
%

\draw[\hatchingcolour, line width=0.4mm] (4,1) to [out=90,in=180] (4.5, 2.25) to [out=0,in=90] (5,1.5) to [out=90,in=180] (5.5, 2.25) to [out=0,in=90] (6,1) to (4,1);

\path[pattern=north west lines, pattern color=\hatchingcolour, line width=0.4mm, opacity=0.5] (4,1) to [out=90,in=180] (4.5, 2.25) to [out=0,in=90] (5,1.5) to [out=90,in=180] (5.5, 2.25) to [out=0,in=90] (6,1) to (4,1);

%
%

\foreach \i in {0,2,8}{
    \node[vertex] at (\i,1) {};
    }

\node[circle,ultra thick,draw=\hatchingcolour,fill=black,inner sep=1.5pt] at (4,1) {};

\node[circle,ultra thick,draw=\hatchingcolour,fill=black,inner sep=1.5pt] at (6,1) {};

%
%

\foreach \i in {1,3,7}{
    \node[vertex] at (\i,1.5) {};
    }

\node[circle,ultra thick,draw=\hatchingcolour,fill=black, inner sep=1.5pt] at (5,1.5) {};

%
%

\node at (-0.7,1) {$L_1(D)$};

\node (dummy) at (8.7,1) {};

\node at (0.5,1.7) {$H_{d,n-1}$};

\node[text=\hatchingcolour] at (4.5,1.7) {$G'$};

\node at (2.75,-0.09) {$u$};

\node at (5.25,-0.09) {$v$};

\node at (4,0.765) {$u'$};

\node at (6,0.765) {$v'$};
\end{tikzpicture}
\caption{\cref{obs : G' hanging off} when $D=G_{d,n}$ and $L_0(D)=\{u,v\}$.}
\label{fig : G' hanging off}
\end{figure}

\begin{obs}\label{obs : H' rooted and opposite}
Let $d\geq 2$ and $n\geq 2$ be integers, let $D \in \{G_{d,n},H_{d,n}\}$, and let $pq\in E(D[L_1(D)])$. Let $w$ be the unique common neighbour of $p$ and $q$ in $L_2(D)$, and let $C$ be the component of $D-\{p,w\}$ that contains $q$. Then, as illustrated in \cref{fig : H' rooted and opposite}, $H':=D-V(C)$ is a copy of $H_{d,n-1}$ such that $L_0(H')=\{p\}$, $L_1(H')=N_D(p)\cap L_2(D)$, and $V(H') \cap L_0(D) =\emptyset$. We say that $H'$ is the copy of $H_{d,n-1}$ in $D$ rooted at $p$ and opposite $q$.
\end{obs}

\begin{figure}[H]
\centering

\newcommand{\blobcolour}{LightGray} 
\newcommand{\hatchingcolour}{red}

\begin{tikzpicture}[vertex/.style={circle,draw,fill=black,inner sep=1.5pt},/tikz/xscale=1.6, /tikz/yscale=1.6]
%
%

\foreach \i in {0,2,4,6}{
    \draw[densely dashed, fill=\blobcolour, opacity=.5] (\i,1) to [out=90,in=180] (\i+0.5, 2.25) to [out=0,in=90] (\i+1,1.5) to (\i,1);
    
    \draw[densely dashed, fill=\blobcolour, opacity=.5] (\i+2,1) to [out=90,in=0] (\i+1.5, 2.25) to [out=180,in=90] (\i+1,1.5) to (\i+2,1);
        
    \draw (\i,1) to (\i+2,1);
    \draw (\i,1) to (\i+1,1.5);
    \draw (\i+2,1) to (\i+1,1.5);
    }

%
%

\node[vertex] (u) at (2.75,0.2) {};
\node[vertex] (v) at (5.25,0.2) {};

\draw (u) to (v);

\draw (u) to (0,1);
\draw (u) to (2,1);
\draw (u) to (4,1);

\draw (v) to (4,1);
\draw (v) to (6,1);
\draw (v) to (8,1);

%
%

\draw[\hatchingcolour, line width=0.4mm] (2,1) to [out=90,in=180] (2.5, 2.25) to [out=0,in=90] (3,1.5) to (2,1);

\path[pattern=north west lines, pattern color=\hatchingcolour, line width=0.4mm, opacity=0.5] (2,1) to [out=90,in=180] (2.5, 2.25) to [out=0,in=90] (3,1.5) to (2,1);

%
%

\foreach \i in {0,4,6,8}{
    \node[vertex] at (\i,1) {};
    }

\node[circle,ultra thick,draw=\hatchingcolour,fill=black,inner sep=1.5pt] at (2,1) {};

%
%

\foreach \i in {1,5,7}{
    \node[vertex] at (\i,1.5) {};
    }

\node[circle,ultra thick,draw=\hatchingcolour,fill=black, inner sep=1.5pt] at (3,1.5) {};

%
%

\node at (-0.7,1) {$L_1(D)$};

\node (dummy) at (8.7,1) {};

\node[text=\hatchingcolour] at (2.5,1.7) {$H'$};

\node at (0.5,1.7) {$H_{d,n-1}$};

\node at (2.75,-0.09) {$u$};
\node at (5.25,-0.09) {$v$};

\node at (2,0.71) {$p$};
\node at (3,1.24) {$w$};
\node at (4,0.71) {$q$};
\end{tikzpicture}
\caption{\cref{obs : H' rooted and opposite} when $D=G_{d,n}$ and $L_0(D)=\{u,v\}$.}
\label{fig : H' rooted and opposite}
\end{figure}

\begin{lem}\label{lem : DTW big lemma}
For any integers $d\geq 2$ and $n\geq 1$:
\begin{enumerate}[label=(\arabic*)]
    \item\label{lem : DTW big lemma (1)} If $(B_x : x\in V(T))$ is a domino tree-decomposition of $H_{d,n}$ and there exists a node $b\in V(T)$ such that
    \[L_0(H_{d,n}) \cup L_1(H_{d,n}) \subseteq B_b,\]
    then $|B_b| \geq n$ or $(B_x : x\in V(T))$ has width at least $d^2-d$.
    \item\label{lem : DTW big lemma (2)} If $(B_x : x\in V(T))$ is a domino tree-decomposition of $G_{d,n}$ and  $L_0(G_{d,n})=\{u,v\}$ and there exists distinct nodes $a,b,c\in V(T)$ such that
    \[u\in B_a,\quad \{u,v\}\subseteq B_b,\quad  v\in B_c,\]
    then $|B_b| \geq n$ or $(B_x : x\in V(T))$ has width at least $d^2-d$. 
\end{enumerate}
\end{lem}

\begin{proof}
Proceed by induction on $n\geq 1$. Both \labelcref{lem : DTW big lemma (1),lem : DTW big lemma (2)} hold in the base case since $|B_b|\geq 1$. Now suppose that $n\geq 2$ and the hypothesis holds for smaller values of $n$. We first prove \labelcref{lem : DTW big lemma (1)}. Let $H:=H_{d,n}$ and proceed by cases depending on the outcomes of \cref{lem : H}.

\textbf{Case 1.i} $L_1(H)$ lies in the intersection of two bags: Then $(B_x : x\in V(T))$ being domino and the edge-property together imply that $L_1(H)\cup N_H(L_1(H))$ is covered by the union of two bags, so $L_0(H)\cup L_1(H)\cup L_2(H)$ is covered by the union of two bags. Then $(B_x : x\in V(T))$ has width at least
\[\tfrac{1}{2}|L_0(H)\cup L_1(H)\cup L_2(H)|-1 = \tfrac{1}{2}(1+d+(d-1)(2d-1))-1=d^2-d,\]
as desired.

\textbf{Case 1.ii} There exists $p\in L_1(H)$ such that $\{p\} \cup (N_H(p) \cap L_2(H)) \subseteq B_b$: Choose $q\in  L_1(H)$ such that $pq\in E(H[L_1(H)])$. As per \cref{obs : H' rooted and opposite}, let $H'$ be the copy of $H_{d,n-1}$ in $H$ rooted at $p$ and opposite $q$. Then $L_0(H') = \{p\}$, $L_1(H')= N_H(p)\cap L_2(H)$, and $V(H')\cap L_0(H) =\emptyset$. Since $(B_x\cap V(H') : x\in V(T))$ is a domino tree-decomposition of $H'$ and $L_0(H')\cup  L_1(H') = \{p\} \cup (N_H(p) \cap L_2(H)) \subseteq B_b\cap V(H')$, by induction $|B_b\cap V(H')| \geq n-1$ or $(B_x\cap V(H') : x\in V(T))$ has width at least $d^2-d$. In the former case, since $V(H')\cap L_0(H)=\emptyset$ and $L_0(H)\subseteq B_b$, $|B_b| \geq n$. In the latter case $(B_x : x\in V(T))$ has width at least $d^2-d$.

\textbf{Case 1.iii} There exists $u'v'\in E(H[L_1(H)])$ and distinct nodes $a',c'\in V(T)\setminus \{b\}$ such that $u'\in B_{a'}$, $\{u',v'\}\subseteq B_b$, and $v'\in B_{c'}$: As per \cref{obs : G' hanging off}, let $G'$ be the copy of $G_{d,n-1}$ in $H$ hanging off $u'v'$. Then $L_0(G')=\{u',v'\}$ and $V(G')\cap L_0(H)=\emptyset$. Since $(B_x\cap V(G') : x\in V(T))$ is a domino tree-decomposition of $G'$ such that $u'\in B_{a'}\cap V(G')$ and $\{u',v'\}\subseteq B_b\cap V(G')$ and $v'\in B_{c'}\cap V(G')$, by induction $|B_b\cap V(G')| \geq n-1$ or $(B_x\cap V(G') : x\in V(T))$ has width at least $d^2-d$. In the former case, since $V(G')\cap L_0(H)=\emptyset$ and $L_0(H)\subseteq B_b$, $|B_b| \geq n$. In the latter case $(B_x : x\in V(T))$ has width at least $d^2-d$.

Next we prove \labelcref{lem : DTW big lemma (2)}. Let $G:=G_{d,n}$ and proceed by cases depending on the outcomes of \cref{lem : G}.

\textbf{Case 2.i} $\{u\}\cup N_G(u) \subseteq B_b$ or $\{v\}\cup N_G(v) \subseteq B_b$: Without loss of generality assume the latter. By definition of $G$, there is a unique common neighbour of $u$ and $v$ in $G$, call it $w$, and let $C$ be the component of $G-\{v,w\}$ that contains $u$. Observe that $H:=G-V(C)$ is a copy of $H_{d,n}$ such that $L_0(H)=\{v\}$. Now since $(B_x\cap V(H) : x\in V(T))$ is a domino tree-decomposition of $H$ such that $L_0(H)\cup L_1(H) = \{v\} \cup (N_G(v)\cap V(H)) \subseteq B_b\cap V(H)$, cases 1.i--iii imply that $|B_b\cap V(H)| \geq n$ or $(B_x\cap V(H) : x\in V(T))$ has width at least $d^2-d$. Hence $|B_b| \geq n$ or $(B_x : x\in V(T))$ has width at least $d^2-d$.

\textbf{Case 2.ii} There exists $p\in L_1(G)$ such that $\{p\} \cup (N_G(p) \cap  L_2(G)) \subseteq B_b$: Choose $q\in  L_1(G)$ such that $pq\in E(G[L_1(G)])$. As per \cref{obs : H' rooted and opposite}, let $H'$ be the copy of $H_{d,n-1}$ in $G$ rooted at $p$ and opposite $q$. Then $L_0(H') = \{p\}$, $L_1(H')= N_G(p)\cap L_2(G)$, and $V(H')\cap L_0(G) =\emptyset$. Since $(B_x\cap V(H') : x\in V(T))$ is a domino tree-decomposition of $H'$ and $L_0(H')\cup  L_1(H') = \{p\} \cup (N_G(p) \cap L_2(G)) \subseteq B_b\cap V(H')$, by induction $|B_b\cap V(H')| \geq n-1$ or $(B_x\cap V(H') : x\in V(T))$ has width at least $d^2-d$. In the former case, since $V(H')\cap L_0(G)=\emptyset$ and $L_0(G)\subseteq B_b$, $|B_b| \geq n$. In the latter case $(B_x : x\in V(T))$ has width at least $d^2-d$.

\textbf{Case 2.iii} There exists $u'v'\in E(G[L_1(G)])$ and distinct nodes $a',c'\in V(T)\setminus \{b\}$ such that $u'\in B_{a'}$, $\{u',v'\}\subseteq B_b$, and $v'\in B_{c'}$: As per \cref{obs : G' hanging off}, let $G'$ be the copy of $G_{d,n-1}$ in $G$ hanging off $u'v'$. Then $L_0(G')=\{u',v'\}$ and $V(G')\cap L_0(G)=\emptyset$. Since $(B_x\cap V(G') : x\in V(T))$ is a domino tree-decomposition of $G'$ such that $u'\in B_{a'}\cap V(G')$ and $\{u',v'\}\subseteq B_b\cap V(G')$ and $v'\in B_{c'}\cap V(G')$, by induction $|B_b\cap V(G')| \geq n-1$ or $(B_x\cap V(G') : x\in V(T))$ has width at least $d^2-d$. In the former case, since $V(G')\cap L_0(G)=\emptyset$ and $L_0(G)\subseteq B_b$, $|B_b| \geq n$. In the latter case $(B_x : x\in V(T))$ has width at least $d^2-d$.

This concludes the proof of the lemma.
\end{proof}

The next result is a precise version of \cref{DTWLowerBound} in the $k=2$ case.

\begin{thm}
\label{OuterplanarLowerBound}
For any integer $d\geqslant 2$, $H_{d,d^2-d+1}$ is an outerplanar graph with maximum degree $2d+4$ and domino treewidth at least $\frac{1}{2}(d^2-d-1)$.
\end{thm}

\begin{proof}
Let $H:=H_{d,d^2-d+1}$ with $L_0(H)=\{v\}$. As illustrated in \cref{OuterplanarConstruction}, 
$G_{d,d^2-d+1}$ is outerplanar and thus $H\subseteq G_{d,d^2-d+1}$ implies $H$ is outerplanar. The following observations imply $\Delta(H)=2d+4$: The vertex in $L_0(H)$ has degree $d$, the vertices in $L_{d^2-d+1}(H)$ have degree at most $4$, and for all $i\in [d^2-d]$ the vertices in $L_i(H)$ have degree at most $2d+4$. Moreover, for any edge $u'v'\in H[L_1(H)]$, the unique vertex in $N_H(u')\cap N_H(v')\cap L_2(H)$ has degree $2d+4$. Hence $\Delta(H)=2d+4$. 

Next, it is shown that $\dtw(H) \geq \frac{1}{2}(d^2-d-1)$. Let $\TT=(B_x:x\in V(T))$ be a domino tree-decomposition of $H$ with minimum width. Update $\mathcal{T}$ by merging the two bags that contain $v$ to a new node, call it $b$. Let $k$ denote the new width of $\mathcal{T}$. Then $\mathcal{T}$ is a domino tree-decomposition of $H$ and $2(\dtw(H)+1) \geq k+1$. Since $v$ only lies in the bag $B_b$, $L_0(H)\cup L_1(H) = \{v\}\cup N_H(v) \subseteq B_b$. Then \labelcref{lem : DTW big lemma (1)} in \cref{lem : DTW big lemma} implies that $k \geq d^2-d$, so $\dtw(H) \geq \frac{1}{2}(d^2-d-1)$, as desired.
\end{proof}

The next lemma determines the domino treewidth of the product of a graph and a complete graph, up to a constant factor. The \defn{cartesian product} of graphs $G_{1}$ and $G_{2}$, denoted by \defn{$G_{1}\square G_{2}$}, is the graph with vertex set $V(G_{1})\times V(G_{2})$ and all the edges of the form $(x_{1},y_{1})(x_{2},y_{2})$, where either $x_{1}=x_{2}$ and $y_{1}y_{2}\in E(G_{2})$, or $x_{1}x_{2}\in E(G_{1})$ and $y_{1}=y_{2}$. 
The \defn{strong product} of $G_{1}$ and $G_{2}$, denoted by \defn{$G_{1}\boxtimes G_{2}$}, is the graph with vertex set $V(G_{1})\times V(G_{2})$ and all the edges of the form $(x_{1},y_{1})(x_{2},y_{2})$, where either $x_{1}=x_{2}$ and $y_{1}y_{2}\in E(G_{2})$, or $x_{1}x_{2}\in E(G_{1})$ and $y_{1}=y_{2}$, or $x_{1}x_{2}\in E(G_{1})$ and $y_{1}y_{2}\in E(G_{2})$.

\begin{lem}
\label{DTWProductCompleteGraph}
For any graph $H$ and any integer $k\geq 1$, 
\begin{enumerate}[label=(\roman*)]
    \item\label{DTWProductCompleteGraph (i)} $k( \dtw(H)+1 ) \leq \dtw(H\square K_{2k-1})+1$, and
    \item\label{DTWProductCompleteGraph (ii)} $\dtw(H\square K_{k})+1 \leq \dtw(H \boxtimes K_{k})+1 \leq k (\dtw(H)+1)$.
\end{enumerate}
\end{lem}

\begin{proof}[Proof of \cref{DTWProductCompleteGraph}\labelcref{DTWProductCompleteGraph (i)}]
Let $G:=H\square K_{2k-1}$ where $V(G) = V(H)\times V(K_{2k-1})$. For each $v\in V(H)$, let $C_v:=\{v\}\times V(K_{2k-1})$. Then $C_v$ is a $(2k-1)$-clique in $G$, and for each edge $uv\in E(H)$, there is a perfect matching between $C_{u}$ and $C_{v}$. Let $\mathcal{T}:=(B_t : t\in V(T))$ be a domino tree-decomposition of $G$ with width $\dtw(G)$. For each $t\in V(T)$, let $B_t':=\{v\in V(H) : |C_v\cap B_t|\geq k\}$. Let $\mathcal{T}':=(B_t' : t\in V(T))$.

We show that $\mathcal{T}'$ is a domino tree-decomposition of $H$. Consider any $v\in V(H)$. Since $C_v$ is a clique in $G$, there exists a node $x\in V(T)$ such that $C_v\subseteq B_x$, so $v\in B_x'$. Now suppose there exists distinct nodes $x,y\in V(T)$ such that $v\in B_x'\cap B_y'$. Then $|C_v\cap B_x|\geq k$ and $|C_v\cap B_y|\geq k$, which along with $|C_v| = 2k-1$ implies that $B_x\cap B_y\not=\emptyset$. Since $\mathcal{T}$ is domino and $B_x\cap B_y\not=\emptyset$, $xy\in E(T)$. This shows that $\{t\in V(T) : v\in B_t'\}$ is a non-empty clique in $T$, implying $v$ is in at most two bags of $\mathcal{T}'$. It remains to show that $\mathcal{T}'$ has the edge-property. Consider any $uv\in E(H)$. Let $P=(x_1, \dots, x_n)$ be a path in $T$ such that $C_u\subseteq B_{x_1}$ and $C_v\subseteq B_{x_n}$. Then 
$|C_{u}\cap B_{x_{1}}|\geq k$. Let $m\in[n]$ be maximum such that $|C_{u}\cap B_{x_{m}}|\geq k$. Then $u\in B_{x_m}'$. If $m=n$, then $\{u,v\}\subseteq B_{x_{n}}'$. If $m<n$, then by the maximality of $m$, $|C_{u}\cap B_{x_{m+1}}|\leq k-1$. Let $S:=C_{u}\setminus B_{x_{m+1}}$. Then $|S|\geq k$. Let $M$ be the perfect matching between $C_u$ and $C_v$, and consider each $ab\in M$ with $a\in S$. Then $a\in B_{x_1}\setminus B_{x_{m+1}}$ and $b\in B_{x_n}$. If $b\not\in B_{x_m}$, then the vertex-property of $\mathcal{T}$ implies that the subtrees of $T$ associated with $a$ and $b$ are disjoint, contradicting the edge-property of $\mathcal{T}$. Thus $b\in B_{x_m}$.
Therefore, $\{b: ab\in M, a\in S\}\subseteq B_{x_m}$ and $|C_v\cap B_{x_m}|\geq |S|\geq k$ 
, implying $\{u,v\}\subseteq B_{x_m}'$. This completes the proof that $\mathcal{T}'$ is a domino tree-decomposition of $H$. Since $(C_v:v\in B_t')$ are pairwise disjoint, 
\[k(\dtw(H)+1)\leq k\max_{t\in V(T)}|B_t'| \leq \max_{t\in V(T)}\sum_{v\in B_t'}|C_v\cap B_t| \leq \max_{t\in V(T)}|B_t|\leq\dtw(G)+1.\qedhere\]
\end{proof}

\begin{proof}[Proof of \cref{DTWProductCompleteGraph}\labelcref{DTWProductCompleteGraph (ii)}]
The first inequality holds since $H\square K_{k}\subseteq H\boxtimes K_{k}$. Let $G:=H\boxtimes K_{k}$. Then $G$ is obtained from $H$ by replacing each vertex $v\in V(H)$ with a clique $C_{v}$ consisting of $k$ new vertices, where each edge $vw$ of $H$ becomes a complete bipartite graph $K_{k,k}$ in $G$ between $C_v$ and $C_w$. Let $(B_{x}:x\in V(T))$ be a domino tree-decomposition of $H$ with width $\dtw(H)$. For each $x\in V(T)$, let $B_{x}':=\bigcup_{v\in B_{x}}C_{v}$. It is clear that $(B_{x}':x\in V(T))$ is a tree-decomposition of $G$. Furthermore, $|B_{x}'|=k|B_{x}|$ for each $x\in V(T)$. And for each $v\in V(H)$, every vertex in $C_{v}$ is in at most two bags of $(B_{x}':x\in V(T))$. Hence, $(B_{x}':x\in V(T))$ is domino with width $k(\dtw(H)+1)-1$.
\end{proof}

We now prove a precise version of \cref{DTWLowerBound}:
\begin{thm}
For all integers $k\geq 1$ and $d\geq 2$ there exists a graph $G$ with $\tw(G)\leq 6k-4$, $\Delta(G) = 2(d+k+1)$, and $\dtw(G)\geq \frac{1}{2}k(d^2-d+1)-1$.
\end{thm}
\begin{proof}
By \cref{OuterplanarLowerBound} there exists an outerplanar graph $H$ with $\Delta(H)=2d+4$ and $\dtw(H)\geq \frac{1}{2}(d^2-d-1)$. Let $G=H\square K_{2k-1}$. Then $\tw(G) \leq (\tw(H)+1)(2k-1)-1= 6k-4$ and $\Delta(G) = \Delta(H)+\Delta(K_{2k-1}) =2(d+k+1)$. Lastly, \cref{DTWProductCompleteGraph} implies that $\dtw(G) \geq k(\dtw(H)+1)-1 = \frac{1}{2}k(d^2-d+1)-1$, as desired.
\end{proof}

\section{Chordal Completions}
\label{ChordalCompletions}

For the proofs in \cref{BoundedSpread} it is helpful to consider chordal graphs, which we now introduce. A graph $G$ is \defn{chordal} if $G$ has no induced cycle of length at least $4$. A \defn{chordal completion} (also called \defn{triangulation}) of a graph $G$ is a chordal graph $G'$ such that $G$ is a spanning subgraph of $G'$. There is a large literature on chordal completions; see the survey by \citet{Heggernes06}. 
A graph is chordal if and only if it has a tree-decomposition in which each bag is a clique (see \citep[Proposition~12.3.6]{Diestel5}). Let \defn{$\omega(G)$} denote the number of vertices in a largest clique in a graph $G$. It follows from the Helly property that $\tw(Q)=\omega(Q)-1$ for every chordal graph $Q$. 

Say $(B_x:x\in V(T))$ is a tree-decomposition of a graph $G$, and $G'$ is obtained from $G$ by adding edges so that $B_x$ is a clique for each $x\in V(T)$. So $(B_x:x\in V(T))$ is a tree-decomposition of $G'$ with the same width, and $G'$ is chordal. It follows that the treewidth of a graph $G$ equals the minimum, taken over all chordal completions $G'$ of $G$, of $\omega(G')-1$. This approach connects chordal completions with small maximum degree to tree-decompositions with small spread. 

\begin{obs}
\label{MakeChordalCompletion}
Given a tree-decomposition of a graph $G$ with width $k$ in which each vertex has spread at most $s$, adding an edge between any two non-adjacent vertices in a common bag, gives a chordal completion of $G$ with treewidth $k$ and maximum degree at most $ks$. Conversely, if a graph $G$ has a chordal completion $G'$ with maximum clique size $k$ and maximum degree $\Delta$, then $G$ has a tree-decomposition with width $k-1$, where each vertex has spread at most $\Delta+1$.
\end{obs}

For a graph $G$ with given treewidth and maximum degree, to prove results about $G$ using a chordal completion $G'$ of $G$, it is necessary that $G'$ has small treewidth (that is, small maximum clique size) and small maximum degree. The following question of \citet{Antony20} naturally arises:

\begin{ques}[\protect\citep{Antony20}] 
\label{AntonyQuestion}
Does every graph with treewidth $k$ and maximum degree $\Delta$ have a chordal completion with treewidth at most $f(k,\Delta)$ and maximum degree at most $g(k,\Delta)$ for some functions $f,g$?
\end{ques}

In the case $k=1$, since every forest is chordal, the answer to \cref{AntonyQuestion} is trivially `yes' with $f(1,\Delta)=1$ and $g(1,\Delta)=\Delta$.

Next consider outerplanar graphs, which have treewidth at most 2. 

For any cycle $C=(x_{0},\ldots,x_{t})$ with $t\geq2$, a \defn{zig-zag triangulation} of $C$ is the graph $G$ obtained from $C$ by adding the edges $\{x_{i}x_{j}:i+j\in\{t,t+1\}\}$ as illustrated in \cref{fig: zig zag}. Then $G$ is a chordal completion of $C$ such that $\deg_{G}(x_{0})=2$, $\deg_{G}(x_{t})\leq3$ and $\deg_{G}(x_{i})\leq4$ for each $i\in[t-1]$.


\begin{figure}[ht]
    \centering
    \begin{tikzpicture}[/tikz/xscale=1.2, /tikz/yscale=1.2]

        \begin{scope}[every node/.style={draw, shape = circle, fill = black, minimum size = 0.1cm, inner sep=1.2pt}, rotate=360/16, transform shape]
            \node (v0) at (-1.2,-1.2) {};
            \node (v1) at (-1.7,0) {};
            \node (v2) at (-1.2,1.2) {};
            \node (v3) at (0,1.7) {};
            \node (v4) at (1.2,1.2) {};
            \node (v5) at (1.7,0) {};
            \node (v6) at (1.2,-1.2) {};
            \node (v7) at (0,-1.7) {};
        \end{scope}

        \begin{scope}[every edge/.style={draw, thick}]
            \path [-] (v0) edge (v1);
            \path [-] (v1) edge (v2);
            \path [-] (v2) edge (v3);
            \path [-] (v3) edge (v4);
            \path [-] (v4) edge (v5);
            \path [-] (v5) edge (v6);
            \path [-] (v6) edge (v7);
            \path [-] (v7) edge (v0);
        \end{scope}

        \begin{scope}[every edge/.style={draw, very thick, LightGray}]
            \path [-] (v7) edge (v1);
            \path [-] (v1) edge (v6);
            \path [-] (v6) edge (v2);
            \path [-] (v2) edge (v5);
            \path [-] (v5) edge (v3);
        \end{scope}
        
        \begin{scope}[scale=1.25, rotate=360/16]
            \node (x0) at (-1.2,-1.2) {$x_{0}$};
            \node (x1) at (-1.7,0) {$x_{1}$};
            \node (x2) at (-1.2,1.2) {$x_{2}$};
            \node (x3) at (0,1.7) {$x_{3}$};
            \node (x4) at (1.2,1.2) {$x_{4}$};
            \node (x5) at (1.7,0) {$x_{5}$};
            \node (x6) at (1.2,-1.2) {$x_{6}$};
            \node (x7) at (0,-1.7) {$x_{7}$};
        \end{scope}
        
    \end{tikzpicture}
    \caption{Zig-zag triangulation of $(x_{0},\ldots,x_{7})$.}
    \label{fig: zig zag}
\end{figure}

\begin{lem}
\label{OuterplanarChordalCompletion}
Every outerplanar graph $G$ has an outerplanar chordal completion $G'$ such that $\deg_{G'}(v)\leq \max\{3\deg_G(v)-2,0\}$ for each vertex $v$. 
\end{lem}

\begin{proof}
Given an outerplanar embedding of $G$, add a zig-zag triangulation on each internal face bounded by a cycle. This produces an outerplanar chordal completion $G'$ of $G$. Each non-isolated vertex $v$ is incident with at most $\deg_G(v)-1$ internal faces. So  $\deg_{G'}(v)\leq \deg_G(v)+2(\deg_G(v)-1)=3\deg_G(v)-2$. 
\end{proof}

Now consider general graphs of treewidth at most 2 (sometimes called \defn{series-parallel}). The next result says that the answer to \cref{AntonyQuestion} in the $k=2$ case is `yes' with $f(2,\Delta)=2$ and $g(2,\Delta)=3\Delta-2$.

\begin{thm}
\label{SeriesParallelChordalCompletion}
Every graph $G$ with treewidth at most 2 has a chordal completion $G'$ with treewidth at most 2 such that $\deg_{G'}(v)\leq\max\{3\deg_G(v)-2,0\}$ for each vertex $v$.
\end{thm}
\begin{proof}

Proceed by induction on $|V(G)|$. If $|V(G)|\leq3$, then $G$ is chordal, so taking $G'=G$ shows the base case. Now suppose that $|V(G)|\geq4$ and the theorem holds for graphs on fewer vertices. If $G$ is disconnected, then applying induction to each component of $G$ gives a desired chordal completion of $G$. Hence it may be assumed that $G$ is connected. A clique $X$ in $G$ is a separator if $G-X$ is disconnected.

\textbf{Case 1.} $G$ has a clique separator of size $2$: Then there exists connected subgraphs $G_1, G_2\subseteq G$ with $G=G_1\cup G_2$, $G_1\cap G_2$ is a copy of $K_2$, and $|V(G_{1})|,|V(G_{2})|<|V(G)|$. Then for each $i\in [2]$, $\tw(G_{i})\leq\tw(G)\leq2$. Let $G_i'$ be the chordal completion of $G_i$ obtained by induction. Then $\tw(G_{i}')\leq2$ and $\deg_{G_{i}'}(v)\leq3\deg_{G_{i}}(v)-2$ for each $v\in V(G_{i}')$. Let $G':=G_1'\cup G_2'$. Since $G_1'\cap G_2'$ is a complete graph, $G'$ is chordal and $\tw(G')=\max\{\tw(G_1'), \tw(G_2')\} \leq 2$. Now consider any $u\in V(G')$. If $u\in V(G_{1}')\setminus V(G_{2}')$, then $\deg_{G'}(u) =\deg_{G_{1}'}(u)\leq3\deg_{G_{1}}(u)-2 =3\deg_{G}(u)-2$. The case of $u\in V(G_{2}')\setminus V(G_{1}')$ is symmetric. If $u\in V(G_{1}' \cap G_{2}')$, then 
\begin{align*}
\deg_{G'}(u)=\deg_{G_{1}'}(u)+\deg_{G_{2}'}(u)-1&\leq(3\deg_{G_{1}}(u)-2)+(3\deg_{G_{2}}(u)-2)-1\\
&=3\deg_{G}(u)-2,
\end{align*}
where the last equality holds since $\deg_{G_{1}}(u)+\deg_{G_{2}}(u)=\deg_{G}(u)+1$.

\textbf{Case 2.} 
$G$ has no clique separator of size $2$: Since $|V(G)|\geq4$, we have $\delta(G)\geq2$ and any degree-$2$ vertex along with its neighbours induces a $3$-vertex path. If $\Delta(G)\leq2$, then $G$ is a cycle and a zig-zag triangulation $G'$ of $G$ is a desired chordal completion of $G$. So we further assume that $\Delta(G)\geq3$.

Let $D$ be the set of degree-$2$ vertices in $G$. It is well-known that the minimum degree of any graph
is at most its treewidth (see \citep[Exercise~12.26]{Diestel5}). Therefore, $2\leq\delta(G)\leq\tw(G)\leq2$ and $D\not=\varnothing$. Since $G$ is connected and $\Delta(G)\geq3$, $G[D]$ is a forest in which each component is a path. Let $\mathcal{P}$ be the set of components in $G[D]$. For each $P\in\mathcal{P}$, let $\widehat{P}:=G[V(P)\cup N_{G}(V(P))]$. Since $G$ has no clique separator of size $2$, each $\widehat{P}$ is an induced path of length at least $2$ and $V(P)$ is the set of inner vertices of $\widehat{P}$. If no two paths in $(\widehat{P}:P\in\mathcal{P})$ have the same set of endpoints, then contracting each $\widehat{P}$ to an edge produces a minor in $G$ with minimum degree at least $3$, and hence treewidth at least $3$. Since the treewidth of any minor of $G$ is at most $\tw(G)$, we have $\tw(G)\geq3$; a contradiction. Hence there exists distinct $\widehat{P},\widehat{Q}$ with the same endpoints, say $x$ and $y$. Let $H$ be the graph obtained from $G$ by contracting $\widehat{P}\cup\widehat{Q}$ to the edge $xy$. Then $\tw(H)\leq\tw(G)\leq2$ and $|V(H)|\leq|V(G)|-2$. Let $H'$ be the chordal completion of $H$ obtained by induction. Then, $\tw(H')\leq2$ and $\deg_{H'}(v)\leq3\deg_{H}(v)-2$ for each $v\in V(H')$. Let $A$ be a zig-zag triangulation of the cycle $\widehat{P}+xy$ such that $\deg_{A}(x)=2$ and $\deg_{A}(y)\leq3$, and let $B$ be a zig-zag triangulation of the cycle $\widehat{Q}+xy$ such that $\deg_{B}(y)=2$ and $\deg_{B}(x)\leq3$ (as shown in \cref{fig: AUnionB}). Let $G':=H'\cup A\cup B$. Since $G'$ is obtained from the chordal graphs $H',A,B$ by pasting along the edge $xy$, $G'$ is a chordal completion of $G$. Moreover, $\tw(G')=\max\{\tw(H'),\tw(A'),\tw(B')\}=2$. Now consider any $u\in V(G')$. If $u\in V(H')\setminus\{x,y\}$, then $\deg_{G'}(u)=\deg_{H'}(u)\leq3\deg_{H}(u)-2=3\deg_{G}(u)-2$. If $u\in V(A\cup B)\setminus\{x,y\}$, then $\deg_{G'}(u)\leq4=3\deg_{G}(u)-2$. If $u\in\{x,y\}$, then $\deg_{A}(u)+\deg_{B}(u)\leq5$ and $\deg_{H}(u)=\deg_{G}(u)-1$, so
\[\deg_{G'}(u)=\deg_{H'}(u)+\deg_{A}(u)+\deg_{B}(u)-2\leq3\deg_{H}(u)+1=3\deg_{G}(u)-2\;.\]
Hence, $G'$ is a desired chordal completion of $G$.
\end{proof}


\begin{figure}[ht]
    \centering
    \begin{tikzpicture}[/tikz/xscale=1.2, /tikz/yscale=1.2]

        \begin{scope}[every node/.style={draw, shape = circle, fill = black, minimum size = 0.1cm, inner sep=1.2pt}]
            \node (x) at (0,0.9) {};
            \node (y) at (0,-0.9) {};
        \end{scope}

        \begin{scope}
            \draw[-, draw, thick] (x) to [out=160,in=90] (-4, 0) to [out=-90,in=200] (y);
            \draw[-, draw, thick] (x) to [out=20,in=90] (4, 0) to [out=-90,in=-20] (y);
            \path[-, draw, thick] (x) edge (y);
        \end{scope}

        \begin{scope}[every node/.style={draw, shape = circle, fill = black, minimum size = 0.3pt, inner sep=1.2pt}]
            \node (v1) at (-1.1,1.2) {};
            \node (v2) at (-1.9,-1.28) {};
            \node (v3) at (-2.7,1.24) {};
            \node (v4) at (-3.5,-0.95) {};

            \node (u4) at (-3.91,0.46) {};
        \end{scope}

        \begin{scope}[every edge/.style={draw, very thick, LightGray}]
            \path [-] (y) edge (v1);
            \path [-] (v1) edge (v2);
            \path [-] (v2) edge (v3);
            \path [-] (v3) edge (v4);
        \end{scope}

        \begin{scope}[every node/.style={draw, shape = circle, fill=black, minimum size = 0.3pt, inner sep=1.2pt}]
            \node (v1) at (1.1,-1.2) {};
            \node (v2) at (1.9,1.28) {};
            \node (v3) at (2.7,-1.24) {};
            \node (v4) at (3.5,0.95) {};

            \node (u4) at (3.91,-0.46) {};
        \end{scope}

        \begin{scope}[every edge/.style={draw, very thick, LightGray}]
            \path [-] (x) edge (v1);
            \path [-] (v1) edge (v2);
            \path [-] (v2) edge (v3);
            \path [-] (v3) edge (v4);
        \end{scope}
        
        \begin{scope}
            \node (x) at (0,0.9+0.4) {$x$};
            \node (y) at (0,-0.9-0.4) {$y$};
            \node (A) at (-2.5,-1.8) {$A$};
            \node (B) at (2.5,-1.8) {$B$};
        \end{scope}
        
    \end{tikzpicture}
    \caption{Drawing of $A\cup B$ in \cref{SeriesParallelChordalCompletion}.}
    \label{fig: AUnionB}
\end{figure}

Given that the answer to \cref{AntonyQuestion} in the $k\leq 2$ case is `yes' with $f(k,\Delta)=k$, it would be tempting to guess the answer is `yes' with $f(k,\Delta)=k$ for any value of $k$. However, this would be false. \citet{DO95} stated (without proof) that for every $n$ there is a graph $G$ with treewidth 3 and maximum degree 4 such that in every tree-decomposition of $G$ with width 3 there is a vertex with spread at least $n$. The next result strengthens this conclusion, and generalises for treewidth at least 3
(using the same graph as \citet{DO95} in the $k=3$ case). In particular, \cref{Factor2} shows that 
to obtain a `yes' answer to \cref{AntonyQuestion} it is necessary that $f(k)\geq 2k-2$ (for $k\geq 3$) regardless of $g(k,\Delta)$.

\begin{prop}
\label{Factor2}
For any integers $k\geq 3$ and $n\geq 2k-1$ there is an $n$-vertex graph $G$ with $\tw(G)= k$ and $\Delta(G)= 2k-2$ such that every chordal completion $G'$ of $G$ with treewidth at most $2k-3$ (that is, $\omega(G')\leq 2k-2$) has maximum degree $n-1$. 
\end{prop}

\begin{proof}
Let $G$ be the graph with $V(G)=\{v,w_1,\dots,w_{n-1}\}$ where $w_iw_j\in E(G)$ if and only if $1\leq |i-j|\leq k-1$, and $N_G(v)=\{w_1,\dots,w_{k-1}\}\cup\{w_{n-k+1},\dots,w_{n-1}\}$. Since $\{v\},\{w_{1}\},\ldots,\{w_{k-1}\},\{w_{k},\ldots,w_{n-1}\}$ are the branch sets of a $K_{k+1}$ minor-model in $G$, $\tw(G)\geq k$. On the other hand, a path-decomposition of $G$ with width $k$ is given by $( \{ w_1,\dots,w_{k},v\}, 
\{ w_2,\dots,w_{k+1},v\}, \dots,
\{ w_{n-k},\dots,w_{n-1},v\})$, so $\tw(G)\leq \pw(G)\leq k$. Hence $\tw(G)=k$. Also note that $\Delta(G)=2k-2$, and that $w_1,\dots,w_{k-2},w_{n-k+2},\dots,w_{n-1}$ are the only vertices with degree at most $2k-3$. 

Consider any chordal completion $G'$ of $G$ with $\omega(G')\leq 2k-2$. So $G'$ has a tree-decomposition $(B_x:x\in V(T))$ with width at most $2k-3$, where each bag is a clique, and no bag is a subset of a neighbouring bag (otherwise merge the bags). Let $x$ be a leaf node of $T$. So there is a vertex $y\in B_x$ that is not in the neighbouring bag of $B_x$. Thus $N_{G'}[y]\subseteq B_x$, implying $\deg_{G'}(y)\leq |B_x|-1 \leq 2k-3$. So $y\in\{w_1,\dots,w_{k-2}\}\cup\{w_{n-k+2},\dots,w_{n-1}\}$ and $v\in N_{G'}(y)$, implying $v\in B_x$. That is, $v$ is in every leaf bag. By the vertex-property of tree-decompositions, $v$ is in every bag. Since every bag is a clique, $v$ is adjacent to every vertex in $G'$.
\end{proof}

\cref{MakeChordalCompletion,Factor2} imply the following lower bound on the width of a tree-decomposition with bounded spread, which strengthens a similar lower bound by \citet{BG26}, who showed that if there are constants $c,c'>0$ such that every graph $G$ has a tree-decomposition of width at most $c \tw(G)$ in which each vertex $v$ has spread at most $c' (\deg(v) + 1)$, then $c\geq 2$.

\begin{cor}
If there is a constant $c$ and a function $f$ such that every graph $G$ with treewidth $k$ and maximum degree $\Delta$ has a tree-decomposition of width at most $ck$ in which each vertex has spread at most $f(k,\Delta)$, then $c\geq 2$. 
\end{cor}

Now consider \cref{AntonyQuestion} for any value of $k$. By \cref{MakeChordalCompletion}. the $O(k\Delta^2)$ bound on the domino treewidth by \citet{Bodlaender-DMTCS99} implies a positive answer to \cref{AntonyQuestion} with  $f(k,\Delta)$ and $g(k,\Delta)$ in $O(k\Delta^2)$. Applying \cref{MakeChordalCompletion} to the tree-decomposition in \cref{SmallSpread}, \citet{Wood25} concluded the following  positive answer to \cref{AntonyQuestion} with $f(k)\in O(k)$ and $g(k,\Delta)\in O(k\Delta)$. 

\begin{thm}[\protect\citep{Wood25}] 
\label{ChordalCompletionMaxDegree}
Every graph $G$ with treewidth $k$ and maximum degree $\Delta$ has a chordal completion with treewidth $14k+13$ and maximum degree  $k(\Delta+1)$. 
\end{thm}

This result is the starting point for our proof in \cref{BoundedSpread}. 

We now show that the bound on the maximum degree in \cref{ChordalCompletionMaxDegree} is best possible up to a constant factor. That is, in \cref{AntonyQuestion}, $g(k,\Delta)\in\Omega(k\Delta)$ regardless of $f(k,\Delta)$.

\begin{prop}\label{CompletionLowerBound}
For any integers $k,d\geq1$, there is a graph $G$ with treewidth $k$ and maximum degree $d+k-1$, such that every chordal completion of $G$ has maximum degree at least $\frac{k+1}{2}d$.
\end{prop}

\begin{proof}
Write $V(K_{1,d})=\{0\}\cup [d]$ so that each $\ell\in [d]$ is a leaf in $K_{1,d}$, and write $V(K_k)=[k]$. Let $G:=K_{1,d}\square K_k$. Then $\Delta(G)=\Delta(K_{1,d})+\Delta(K_k) = d+k-1$ and $V(G) = (\{0\}\cup [d])\times [k]$.

The following shows that $\tw(G)=k$: For each $\ell\in [d]$, let $\mathcal{T}_\ell$ be the path-decomposition of $G[\{0,\ell\}\times [k]]$ given by
\begin{align*}
\mathcal{T}_\ell=\big(\{(0,1), (\ell,1), \dots, (\ell,k)\}, &\;\{(0,1), (0,2), (\ell, 2), \dots, (\ell,k)\},\dots\\
&\dots, \{(0,1), \dots, (0,k), (\ell, k)\}, \{(0,1), \dots, (0,k)\}\big).
\end{align*}
Since $G=\bigcup_{\ell\in [d]} G[\{0,\ell\}\times [k]]$ and each $\mathcal{T}_\ell$ has $\{0\}\times [k]$ as its last bag, identifying the last bag of each of the $\mathcal{T}_\ell$'s produces a tree-decomposition of $G$ with width $k$, so $\tw(G)\leq k$. On the other hand, since $\{(0,1)\}, \dots, \{(0,k)\}, \{1\}\times [k]$ are the branch sets of a $K_{k+1}$ minor-model in $G$, $\tw(G)\geq k$. Hence $\tw(G)=k$.


Now consider any chordal completion $H$ of $G$. For each $\ell\in [d]$ and $i,j\in [k]$ with $i<j$, let $C_{\ell,i,j}:=G[\{(0,i),(\ell,i),(\ell,j),(0,j)\}]$ and let $\overline{C}_{i,j,\ell}$ be its complement graph. Since $C_{\ell,i,j}$ is a chordless $4$-cycle in $G$, there exists $e_{\ell,i,j}\in E(\overline{C}_{i,j,\ell})\cap E(H)$. We claim that the sets in $(E(\overline{C}_{i,j,\ell}) : \ell\in [d],\, i,j\in [k], i<j)$ are pairwise disjoint, hence the edges in $(e_{\ell,i,j} : \ell\in [d],\, i,j\in [k], i<j)$ are pairwise distinct: If $E(\overline{C}_{\ell',i',j'})\cap E(\overline{C}_{\ell,i,j})\not=\emptyset$, then $C_{\ell',i',j'}$ and $C_{\ell,i,j}$ have a common pair of non-adjacent vertices,
so $\{(0,i),(\ell,j)\}=\{(0,i'), (\ell',j')\}$, or $\{(0,i),(\ell,j)\}=\{(0,j'), (\ell',i')\}$, or $\{(0,j),(\ell,i)\}=\{(0,i'), (\ell',j')\}$, or $\{(0,j),(\ell,i)\}=\{(0,j'), (\ell',i')\}$. 
Since $\ell$ and $\ell'$ are positive, all cases imply $\ell=\ell'$ and $\{i,j\}=\{i',j'\}$. 
Since $i<j$ and $i'<j'$, 
$i=i$ and $j=j'$. Hence $(\ell,i,j)=(\ell',i',j')$ and the desired claim holds. The number of edges in $H$ between $\{0\}\times [k]$ and $[d]\times [k]$ is at least $|\{e_{\ell,i,j} : \ell\in [d], i,j\in [k], i<j\}|+kd=\binom{k}{2}d+kd$. So for some $i\in [k]$, $\Delta(H) \geq \deg_H((0,i))\geq \frac{1}{k}(\binom{k}{2}d+kd) = \frac{k+1}{2}d$.
\end{proof}

\section{\boldmath Bounded Spread and \texorpdfstring{$O(\Delta)$}{O(∆)} Width}
\label{BoundedSpread}

This section constructs a tree-decomposition of a graph $G$ with width $O(\tw(G)\,\Delta(G))$ and spread bounded by a function of $\tw(G)$. First consider outerplanar graphs. The spread bound of 3 in the following result is optimal because of \cref{OuterplanarLowerBound}.

\begin{lem}
\label{EdgeMaximalOuterplanar}
Every chordal outerplanar graph $G$ with maximum degree $\Delta\geq1$ has a tree-decomposition with width at most $2\Delta-1$ and spread at most 3. 
\end{lem}

\begin{proof}
We may assume that $G$ is connected. Fix a vertex $r$ of $G$. For each $i\geq0$, let $V_{i}:=\{v\in V(G):\dist_{G}(v,r)=i\}$. So $V_{0}=\{r\}$.

For $i\geq1$, consider a component $H$ of $G[V_{i}]$. It is a well-known that the class of outerplanar graphs is minor-closed, and that $K_4$ and $K_{2,3}$ are not outerplanar, hence $K_4$ and $K_{2,3}$ are not minors of $G$. If $H$ contains a cycle $C$, then contracting $C$ to $K_{3}$ and contracting $G[V_{0}\cup\cdots\cup V_{i-1}]$ to a vertex gives a $K_{4}$-minor in $G$, a contradiction. So $H$ is a tree. If $\Delta(H)\geq3$, then contracting $G[V_{0}\cup\cdots\cup V_{i-1}]$ to a vertex gives a $K_{2,3}$-minor in $G$, a contradiction. Hence $H$ is a path, implying $G[V_{i}]$ is a disjoint union of paths. 

Let $K$ be the set of vertices in $V_{i-1}$ with at least one neighbour in $H$. We say $H$ is a \defn{child component} of $K$. Note that $H$ is dominated by $K$. Since $G$ is connected, $K\not=\varnothing$. We claim that $K$ is a clique with $|K|\in\{1,2\}$. If $i=1$, then $K=\{r\}$, so we may assume that $i\geq2$. Suppose for a contradiction that there are distinct $x,y\in K$ with $xy\not\in E(G)$. Let $P$ be a shortest path in $G[V_{0}\cup\cdots\cup V_{i-2}]$ between $N_{G}(x)\cap V_{i-2}$ and $N_{G}(y)\cap V_{i-2}$, and let $Q$ be a shortest path in $H$ between $N_{G}(x)\cap V(H)$ and $N_{G}(y)\cap V(H)$. Since $xy\not\in E(G)$ and there are no edges in $G$ between $P$ and $Q$, this gives a chordless cycle of length at least 4 in $G$, a contradiction. So $K$ is a clique. Furthermore, $|K|\leq2$, as otherwise contracting $H$ to a vertex gives a $K_{4}$-minor in $G$, contradicting the outerplanarity of $G$.

For each $i\geq0$, since each component $P$ of $G[V_{i}]$ is a path, we can choose one of the endpoints $u_{P}$ of $P$ and view $P$ as a path rooted at $u_{P}$. 
For each $w\in V(P)$, let $Y_{w}$ be the set containing $w$ and (if it exists) the child of $w$. Note that $Y_{w}$ is well-defined since $w$ is in exactly one component of exactly one $G[V_{i}]$. 
Charge each vertex in each child component of $\{u_{P}\}$ to $u_{P}$, and for each edge $vw$ of $P$, where $w$ is the parent of $v$, charge each vertex in each child component of $\{v,w\}$ or $\{v\}$ to $w$. Then each vertex in $V(G)\setminus\{r\}$ is charged to exactly one other vertex of $G$, and $r$ is not charged to any vertex.

It is clear that the graph $T$ with vertex set $V(T)=V(G)$ and edge set $E(T)=\{vw:v\text{ is charged to }w\text{, or }w\text{ is charged to }v\}$ is a tree. For each $x\in V(T)$, let $B_{x}:=\{z:z\text{ is charged to }x\}\cup Y_{x}$. We claim that $\mathcal{T}:=(B_{x}:x\in V(T))$ is a tree-decomposition of $G$ with width at most $2\Delta-1$ and spread at most $3$. Note that $r$ only appears in the bag $B_{r}$. 
Let $x\in V(G-r)$. Then $x$ is charged to exactly one vertex, say $y$. Then $xy\in E(T)$. Let $i\geq1$ such that $x$ is in a component $P$ of $G[V_{i}]$. Then $x$ is only in the bags in $\{B_y\}\cup \{B_w : w\in V(P), x\in Y_w\}$. Suppose $x$ is the root vertex $u_P$ of $P$, then $\{w\in V(P): x\in Y_w\}=\{x\}$, so $x$ is only in the bags $B_x$ and $B_y$. Otherwise, $x$ has a parent in $P$, say $z$. Then $\{w\in V(P) : x\in Y_w\} = \{x,z\}$, so $x$ is only in the bags $B_x, B_y, B_z$. Moreover, since $x$ is charged to $y$, every vertex in $V(P)$ is charged to $y$, so $yz\in E(T)$, and $\{x,y,z\}$ is connected in $T$. This shows the vertex-property and the spread is at most $3$. Let $xy\in E(G)$. Then $x\in V_{i}$ and $y\in V_{j}$ for some $i\geq j\geq0$. Let $P$ be the component of $G[V_{j}]$ containing $y$. If $i=j$, then $xy\in E(P)$; say $y$ is the parent of $x$ in $P$ without loss of generality. Then $x\in Y_{y}$, implying $\{x,y\}\subseteq B_{y}$. Otherwise, $i=j+1$. If $y=u_{P}$, then $x$ is charged to $y$, implying $\{x,y\}\subseteq B_{y}$. If $y\not=u_{P}$, then $y$ has a parent $z$ in $P$, and $x$ is charged to $z$, implying $\{x,y\}\subseteq B_{z}$. This shows the edge-property.
It remains to check the width. Let $x\in V(T)$. Then $Y_{x}$ is a clique in $G$ with $|Y_{x}|\in\{1,2\}$ and $B_{x}\subseteq N_{G}(Y_{x}) \cup  Y_{x}$. If $|Y_{x}|=1$ then $|B_{x}|\leq\Delta+1\leq2\Delta$. If $|Y_{x}|=2$, then $Y_{x}=\{x,y\}$ for some $y\in V(G)$. Since $B_x \subseteq N_G[x]\cup N_G[y]$, $|B_x| \leq 2(\Delta+1) - 2=2\Delta$.
Hence $\mathcal{T}$ is a desired tree-decomposition of $G$.
\end{proof}

\cref{OuterplanarChordalCompletion,EdgeMaximalOuterplanar} imply that every outerplanar graph $G$ with maximum degree $\Delta\geq 1$ has a tree-decomposition with width at most $6\Delta(G)-5$ and spread at most 3, which proves \cref{OuterplanarUpperBound}. 

Now consider general graphs of given treewidth and maximum degree (\cref{LinearWidthBoundedSpread} from \cref{Intro}). As in the outerplanar case, it is helpful to consider chordal completions. An orientation of a graph $G$ is \defn{simplicial} if it is acyclic, and for each $v\in V(G)$ the in-neighbourhood of $v$ is a clique. It is well-known that a graph has a simplicial orientation if and only if it is chordal~\citep{Rose70}.

\begin{lem}
\label{kTree}
Let $G$ be a chordal graph with a fixed simplicial orientation. For each source vertex $s$ in $G$, there is a tree-decomposition of $G$ with width at most $\Delta^+(G)$, such that each vertex $v$ in $G$ has spread at most $2^{\deg^-_G(v)}$, and some bag equals $N_G[s]$.
\end{lem}

\begin{proof}
We proceed by induction on $|V(G)|$. In the base case, if $V(G)=N_G[s]$, then the tree-decomposition of $G$ with one bag $V(G)$ satisfies the claim since $2^0=1$. 
Now assume that $V(G)\setminus N_G[s]\neq \emptyset$. We may assume that $G$ is connected. 

Consider a component $X$ of $G-N_G[s]$. Let $C_X$ be the set of vertices in $N_G(s)$ adjacent to some vertex in $X$. Since $G$ is chordal and connected, $C_X$ is a non-empty clique. Say $X$ \defn{belongs} to the source vertex of $C_X$. 

For each vertex $v\in N_G(s)$, let $G_v$ be the graph induced by the union, taken over all components $X$ of $G-N_G[s]$ that belong to $v$, of $V(X)\cup C_X$. Let $G_v$ inherit its $k$-simplicial orientation from $G$. Note that if $xy$ is an edge of $G$ with $x\in V(X)$ and $y\in C_X$, then since the in-neighbourhood of $y$ is a clique and $sx\not\in E(G)$, it must be that $xy$ is oriented towards $x$. In particular, $v$ is a source of $G_v$. 

For each $v\in N_G(s)$, apply induction to $G_v$ with source $v$. We obtain a tree-decomposition $\mathcal{T}_v$ of $G_v$ with width at most $\Delta^+(G_v)\leq\Delta^+(G)$, such that each vertex $x$ in $G_v$ has spread at most $2^{\deg^-_{G_v}(x)}$, 
and some bag equals $N_{G_v}[v]$. 

We may assume that the trees indexing $(\mathcal{T}_v:v\in N_G(s))$ are pairwise disjoint. Add a new bag $N_G[s]$ adjacent to the bag $N_{G_v}[v]$ of $G_v$, for each $v\in N_G(s)$. We obtain a tree-decomposition $\mathcal{T}$ of $G$ with width at most $\max\{\Delta^+(G),|N_G[s]|-1\}=\Delta^+(G)$, in which some bag equals $N_G[s]$. 

We now bound the spread. By construction, $s$ has spread $1=2^0=2^{\deg^-(s)}$. 
Each vertex $x\in V(G)\setminus N_G[s]$ is in exactly one subgraph $G_v$ (with $v\in N_G(s)$), so $x$ has spread at most 
$2^{\deg^-_{G_v}(x)}=2^{\deg^-_G(x)}$. 

Consider a vertex $w\in N_G(s)$. Let $N^-_G[w]=\{v_0,v_1,\dots,v_d\}$ be the closed in-neighbourhood of $w$, where $v_iv_j$ is oriented from $v_i$ to $v_j$ whenever $i<j$. So $v_0=s$ and $v_{d}=w$ and $d=\deg_G^-(w)$. 
If $w\in V(G_v)$ for some $v\in N_G(s)$, then $w\in C_X$ for some component $X$ of $G-N_G[s]$ that belongs to $v$, implying that $v\in\{v_1,\dots,v_d\}$ (since $v\neq s$). 
Moreover, $w$ has indegree at most $d-i$ in $G_{v_i}$. 
By induction, $w$ has spread at most 
$2^{d-i}$ in $\mathcal{T}_{v_i}$. 
So $w$ has spread at most 
$1+\sum_{i=1}^{d}2^{d-i}=2^d=2^{\deg^-_G(w)}$ in $\mathcal{T}$, as claimed. 
\end{proof}

We have the following precise version of \cref{LinearWidthBoundedSpread}. 

\begin{thm}
\label{LinearWidthBoundedSpreadPrecise}
Every graph $G$ with treewidth $k$ and maximum degree $\Delta$ has a tree-decomposition with width at most $k(\Delta+1)$, such that each vertex in $G$ has spread at most $2^{14(k+1)}$.
\end{thm}

\begin{proof}
By \cref{ChordalCompletionMaxDegree}, there is a chordal completion $G'$ of $G$ with $\omega(G')\leq 14(k+1)$ and maximum degree at most $k(\Delta+1)$. \cref{kTree} implies that $G'$ and thus $G$ has a tree-decomposition with width at most $k(\Delta+1)$, such that each vertex has spread at most $2^{\omega(G')}\leq 2^{14(k+1)}$. 
\end{proof}


We finish with an open problem. \cref{LinearWidthBoundedSpread} says that every graph $G$ with treewidth $k$ and maximum degree $\Delta$ has a tree-decomposition with width $O_k(\Delta)$, such that each vertex has spread at most $f(k)$, where $f(k)\in 2^{O(k)}$. What is the least function $f$ in such a result? Our results do not exclude the possibility that $f(k)$ is at most some absolute constant, possibly 3.

{\fontsize{11pt}{12pt}\selectfont
\def\soft#1{\leavevmode\setbox0=\hbox{h}\dimen7=\ht0\advance \dimen7
  by-1ex\relax\if t#1\relax\rlap{\raise.6\dimen7
  \hbox{\kern.3ex\char'47}}#1\relax\else\if T#1\relax
  \rlap{\raise.5\dimen7\hbox{\kern1.3ex\char'47}}#1\relax \else\if
  d#1\relax\rlap{\raise.5\dimen7\hbox{\kern.9ex \char'47}}#1\relax\else\if
  D#1\relax\rlap{\raise.5\dimen7 \hbox{\kern1.4ex\char'47}}#1\relax\else\if
  l#1\relax \rlap{\raise.5\dimen7\hbox{\kern.4ex\char'47}}#1\relax \else\if
  L#1\relax\rlap{\raise.5\dimen7\hbox{\kern.7ex
  \char'47}}#1\relax\else\message{accent \string\soft \space #1 not
  defined!}#1\relax\fi\fi\fi\fi\fi\fi}

}
\end{document}